\documentclass[10pt]{article_ET_m}

\usepackage[dvipdf]{graphicx}
\usepackage[french,english]{babel}
\usepackage[utf8]{inputenc}
\usepackage[T1]{fontenc} 
\usepackage{latexsym}
\usepackage{amsmath,amssymb,amsfonts,amsthm}
\usepackage{xcolor}
\usepackage{mathrsfs}
\usepackage{bbm}
\usepackage{bm}
\usepackage{enumitem}
\usepackage{mathtools}

\usepackage[numbers,comma,sort]{natbib}

\definecolor{db}{RGB}{0, 0, 130}
\usepackage[colorlinks=true,citecolor=red,linkcolor=db,urlcolor=blue,pdfstartview=FitH]{hyperref}

\usepackage{pdfsync}

\numberwithin{equation}{section}

\definecolor{rp}{rgb}{0.25, 0, 0.75}
\definecolor{dg}{rgb}{0, 0.5, 0}

\newcommand{\R}{\mathbb{R}}

\newcommand{\N}{\mathbb{N}}
\newcommand{\NN}{\mathcal{N}}
\newcommand{\NNN}{\widetilde{\mathcal{N}}}
\newcommand{\MM}{\mathcal{M}}
\newcommand{\E}{\mathbb{E}}
\newcommand{\PP}{\mathbb{P}}
\newcommand{\CC}{\mathcal{C}}
\newcommand{\dd}{\mathrm{d}}   

\newcommand{\pa}{\partial}
\newcommand{\1}{\mathbbm{1}}
\newcommand{\eps}{\varepsilon}
\DeclareMathOperator*{\esssup}{ess\,sup}

\usepackage{todonotes}

\makeatletter
\newcommand{\customlabel}[2]{%
   \protected@write \@auxout {}{\string\newlabel {#1}{{#2}{\thepage}{#2}{#1}{}}}%
   \hypertarget{#1}{#2\hspace{-0.14cm}}
}
\makeatother

\newtheorem{Assumption}{Assumption}
\newtheorem{theorem}{Theorem}
\newtheorem*{acknowledgement}{Acknowledgements}
\newtheorem{definition}{Definition}[section]

\newtheorem{corollary}[definition]{Corollary}

\newtheorem{lemma}[definition]{Lemma}

\newtheorem{proposition}[definition]{Proposition}
\newtheorem{remark}[definition]{Remark}

\author{Thomas Cavallazzi\footnote{Universit\'e Paris-Saclay, CentraleSup\'elec and CNRS FR-3487, France. \texttt{thomas.cavallazzi@centralesupelec.fr}.} \and 
Alexandre Richard\footnote{Universit\'e Paris-Saclay, CentraleSup\'elec and CNRS FR-3487, France. \texttt{alexandre.richard@centralesupelec.fr}.}
\and Milica Toma\v sevi\'c\footnote{CMAP, Ecole polytechnique, CNRS, I.P. Paris, 91128 Palaiseau, France. ~\hfill ~\newline\texttt{milica.tomasevic@polytechnique.edu}.
}}

\title{ \Large{\textbf{Quantitative approximation of a Keller--Segel PDE by a branching moderately interacting particle system and suppression of blow-up}}}

\begin{document}

\maketitle

\begin{abstract}
The Keller--Segel PDE is a model for chemotaxis known to exhibit possible finite-time blow-up. Following a seminal work by Tello and Winkler~\cite{TelloWinkler}, a logistic damping term is added in this PDE and local well-posedness of mild solutions is proven. When the space dimension is $2$ or when the damping is strong enough, the solution is global in time. 
In the second part of this work, a microscopic description of this model is introduced in terms of a system of stochastic moderately interacting particles. This system features two main characteristics: the interaction between particles happens through a singular (Coulomb-type) kernel which is attractive; and the particles are subject to demographic events, birth and death due to local competition with other particles. The latter induces a branching structure of the particle system. Then the main result of this work is the convergence of the empirical measure of the particle system towards the  Keller--Segel PDE with logistic damping, with a rate of order $N^{-\frac{1}{2(d+1)}}$. 
\end{abstract}


\medskip

\noindent\textbf{MSC2020 subject classification: } 60K35, 60H30, 35K55, 35Q70.

\setcounter{tocdepth}{2}
\renewcommand\contentsname{}
\vspace{-1cm}

\tableofcontents

\section{Introduction}

The Keller--Segel model is a system of partial differential equations describing the collective motion of cells, typically bacteria, attracted by a chemical substance they themselves emit. This chemotactic behavior has been studied extensively from both analytical and probabilistic viewpoints. The parabolic-elliptic form of the Keller--Segel system is known to exhibit remarkable phenomena such as pattern formation and finite-time blow-up, depending on the space dimension and the total mass of the initial data, see \emph{e.g.} \cite{KellerSegel,Nagai11,perthame,Blanchet} and references therein. It reads
\begin{equation*}
\left\{
\begin{array}{rcl}
	\partial_t u_t(x) &=& \Delta u_t(x) - \chi \nabla \cdot \big(u_t(x)\, \nabla c_{t} (x)\big), \quad t>0, \, x \in \R^d,\\[0.3em]
	- \Delta c_t(x) &=& u_t(x), \quad t>0, \, x \in \R^d,\\[0.3em]
	u(0,x) &=& u_0(x),
\end{array}
\right.
\end{equation*}
where $\chi>0$ is the intensity of the chemo-attractant. This system is a parabolic-elliptic coupling: the density $u$ diffuses while being attracted by the potential $c$, which models the concentration of chemo-attractant produced by the particles.
For the solution $c$ to the above Poisson equation, it holds
\begin{align}\label{KS_expression_gradient_concentration}
	\nabla c_t = K \ast_{x} u_t,
\end{align}
 where the interaction kernel $K$ is defined by
\begin{equation}\label{eq:defK}
	K(x) \coloneqq -\frac{1}{C_d}\frac{x}{|x|^d},
\end{equation}
with $C_d = d |B_1(0)|$ denoting the volume of the unit ball multiplied by $d$.
Since $K$ is the gradient of the fundamental solution of the Poisson equation, it satisfies equivalently
\begin{align}\label{eq:divK}
	 - \nabla \cdot K = \delta_{0}.
\end{align}
From a PDE perspective, the blow-up corresponds informally to an aggregation of the density into Dirac masses, reflecting the strong attraction induced by the nonlocal drift term.

\medskip

Several mechanisms have been proposed to prevent or delay blow-up in the Keller--Segel system. One natural way to regularize the dynamics is to introduce a logistic source term, modeling birth and death processes that counteract aggregation. In their seminal work, Tello and Winkler~\cite{TelloWinkler} studied the Keller--Segel system with a logistic damping term under Neumann boundary conditions on a bounded domain. They proved global existence and boundedness of classical solutions when the damping is strong enough, showing that the logistic reaction can effectively suppress blow-up. We briefly mention variations on the Keller--Segel model that lead to prevention of blow-up: one can modify the chemotactic sensitivity~\cite{BurgerEtAl} or cell diffusion to account for the volume-filling effect (particles slowing down in crowded regions), see \emph{e.g.}~\cite{CalvezCarrillo}; combining logarithmic chemotactic sensitivity with logistic source to prevent blow-up has been investigated in~\cite{FujieEtAl}; 
cross-diffusion has also been shown to prevent blow-up~\cite{HittmeirJungel,CarrilloHittmeirJungel}; 
 finally we mention that suppression of blow-up in Keller--Segel has been observed through the addition of convection terms or external flows, which can redistribute mass and prevent concentration, see \cite{KiselevXu,IyerXuZlatos,BedrossianHe} for recent developments. Blow-up control through logistic-type reaction terms has also been studied in the case of Dirac-type interactions \cite{ChenDesvillettesLatos}. 

In the present work, we first adapt and extend these ideas to an initial value problem on the whole space $\R^d$, $d \ge 2$, with non regular initial condition. More precisely, for bounded initial conditions, we consider the mild solutions of the following Keller--Segel equation with logistic source:
\begin{equation}\label{eq:KS_EDP}
\left\{
\begin{array}{rcl}
	\partial_t u_t(x) &=& \Delta u_t(x) - \chi \nabla \cdot \big(u_t(x)\, K \ast_x u_t(x)\big)
	+ \nu u_t(x) - \mu u_t^2(x), \quad t>0, \, x \in \R^d,\\[0.3em]
	u(0,x) &=& u_0(x),
\end{array}
\right.
\end{equation}
where $\chi,\mu>0$ and $\nu\ge0$, for which we prove local well-posedness, see Proposition~\ref{prop:explosion-criterion}.
 The competition between the nonlinear aggregation and the logistic damping term $-\mu u^2$ determines whether the system exhibits blow-up or global regularity. In our first main result, Theorem~\ref{Thm_WP_KS}, we show that for $d=2$ or for $\mu$ large enough, the mild solution to \eqref{eq:KS_EDP} is global in time.

\medskip

A major line of research aims to derive such macroscopic PDEs from microscopic particle systems. For the Keller--Segel model, \emph{a first important feature} is the singular and attractive nature of the kernel. 
Several works have established the mean field and propagation of chaos limits for stochastic particle systems interacting through the singular kernel $K$, depending on the intensity $\chi$ of the chemo-attraction, see \emph{e.g.} the works of Fournier and Jourdain~\cite{FournierJourdain}, Bresch, Jabin and Wang~\cite{BJW23} and Tardy~\cite{Tardy}. Uniform-in-time propagation of chaos results have also been obtained for related attractive models such as the logarithmic gas by Chodron de Courcel et al.~\cite{ChodronEtAl}. A main difficulty that arises in these particle systems is the possibility of collisions of particles when $\chi$ is sufficiently large, see Fournier and Tardy~\cite{FournierTardy} where such collisions are described precisely.



In the present work, to connect the microscopic and macroscopic descriptions, we will consider moderately interacting particle systems, where the interaction kernel is regularized at the scale of the number of particles. This procedure departs from the usual mean field description and corresponds to interactions that are more local. This regime has been extensively studied, initially by Oelschl\"ager \cite{Oelschlager85,Oelschlager87}, then by Méléard and Roelly~\cite{MeleardRoelly}, Jourdain and Méléard~\cite{JourdainMeleard}, and Stevens for the Keller--Segel model~\cite{Stevens}; with recent quantitative extensions and convergence rates obtained in Chen et al.~\cite{ChenHolzingerHuo,ChenHolzingerJungel}. The semigroup approach developed by Flandoli, Leimbach and Olivera~\cite{FlandoliLeimbachOlivera} provided an alternative route to approximate nonlinear PDEs such as Fisher--KPP. This approach was later refined in~\cite{Pisa,ORT-Burgers} to handle singular kernels, permitting to obtain rates of convergence in strong topologies such as in Lebesgue or Bessel spaces. Extensions to kinetic and stable jump models were investigated in~\cite{HJMRZ}, and the impact of common noise was studied in~\cite{KnorstOliveraSouza}. A numerical approach to the Keller--Segel and related Coulomb-type interactions dynamics has also been analysed by Cazacu~\cite{Cazacu}.

\medskip

\emph{A second important feature} of the microscopic model we consider here is the presence of demographic events, modelled by a branching mechanism. Unlike the models described in the previous paragraphs in which all particles were alive at all time (except~\cite{FlandoliLeimbachOlivera}), particles are now likely to die or give birth to other particles. This approach to interpret the logistic term in~\eqref{eq:KS_EDP} in terms of demographic events had already appeared in the work of Chauvin et al.~\cite{ChauvinEtAl}, where propagation of chaos and fluctuations were proven; and in Régnier and Talay~\cite{RegnierTalay}, where a convergence rate for a numerical approximation was achieved. More recently, Fontbona and Méléard~\cite{FontbonaMeleard} proved propagation of chaos when multiple species are competing, deriving at the limit a nonlocal Lotka--Volterra cross-diffusion system. This was further investigated by Fontbona and Mu\~noz~\cite{FontbonaMunoz} who obtained an optimal rate of convergence of the empirical measure towards the PDE, when the death rate of particles is global (instead of local herein); see also Claisse et al.~\cite{ClaisseKangTan} for related branching McKean--Vlasov diffusions.

\medskip

The works mentioned in the previous paragraph cover regular interactions. Studying the propagation of chaos for a model with both singular interactions and branching seems to be new in the literature. Beyond this novelty, the main motivation to couple these two features is to show evidence of the phenomenon of suppression of blow-up, \emph{i.e.} to highlight that propagation of chaos may hold on any time interval, while for the particle system without branching, propagation of chaos would only hold on a sufficiently small time interval.

In the present framework, a particle system starts with $N$ initial particles. A particle $k$ alive at time $t$ has appeared in the system at time $T_{0}^k$, which is either $0$ if the particle originates from the initial system, or it is some random time when it was given birth by a mother particle. It will disappear from the system at time $T_{1}^k$, either by giving birth to $2$ new particles with rate $\nu$, or dying at a rate that is proportional to the local concentration of particles around its position $X^{k,N}_{t}$ (see Section~\ref{subsubsec:constructionIPS}). Between these two times, the particle evolves according to the following Keller--Segel dynamics, interacting with other particles alive at time $t$ (denoted by $I_{t}$):
\begin{equation*}
\dd X^{k,N}_{t} = \chi F\Big(\frac{1}{N} \sum_{j \in I_t} K\ast \theta^N (X^{k,N}_t - X^{j,N}_t) \Big)\, \dd t + \sqrt{2}\, \dd B^k_t, \quad t \in [T^k_0, T^k_1),
\end{equation*}
where the $(B^k)_{k}$ are independent standard Brownian motions, $(\theta^N)_{N\in \N^*}$ is a mollifying sequence which accounts for the moderate interaction regime, and $F$ is a cutoff function which appears for technical reasons. At the limit (\emph{i.e.} in the PDE), this cutoff does not appear as it would be applied on bounded terms, so choosing the cutoff large enough makes it transparent for the limit equation.
We are interested in the convergence of the empirical measure of the system
\begin{equation*}
\mu^N_t \coloneqq \frac{1}{N} \sum_{k\in I_{t}} \delta_{X^{k,N}_{t}}
\end{equation*}
as the number of initial particles $N$ tends to infinity. Our approach will consist in mollifying $\mu^N$, namely by introducing $u^N \coloneqq \theta^N \ast \mu^N$, and compare $u^N$ to the mild solution $u$ of the PDE~\eqref{eq:KS_EDP} given by Theorem~\ref{Thm_WP_KS}. The convergence is shown to hold on the time interval $[0,T]$, for any $T$ up to the time of existence given by Theorem~\ref{Thm_WP_KS}. Hence the convergence stated in Theorem~\ref{th:convergence_moderate} reads
\begin{equation*}
	\sup_{t \in [0,T]} \Big( \E \big\Vert u^N_t - u_t \big\Vert_{L^1 \cap L^r(\R^d)}^m \Big)^\frac{1}{m} 
	\lesssim \Big( \E \big\Vert u^N_0 - u_0 \big\Vert_{L^1 \cap L^r(\R^d)}^m \Big)^\frac{1}{m} + N^{-\varrho}, 
\end{equation*}
with a rate $\varrho$ that depends explicitly on the dimension $d$, the integrability parameter $r$ (chosen larger than $d$) and the choice of mollifier $\theta^N$. We refer to Remark~\ref{rk:rate} for a discussion on these parameters, and the conclusion that $\varrho$ can be close to $\frac{1}{2(d+1)}$. An extension of this result permits to obtain the same rate of convergence on the Kantorovich-Rubinstein distance between the genuine empirical measure $\mu^N_{t}$ and $u_{t}$, see Corollary~\ref{cor:KR-convergence}.

\paragraph{Organisation of the paper.}
In Section~\ref{sec:prelim_and_main}, the setting and notations used throughout of the paper are introduced. The main results, Theorem~\ref{Thm_WP_KS} and Theorem~\ref{th:convergence_moderate}, are stated respectively in Subsection~\ref{subsec:WP_PDE} and Subsection~\ref{subsec:ThmConv}. The proof of the former is carried out in Section~\ref{sec:PDE}. 
Finally in Section~\ref{sec:particles}, we prove Theorem~\ref{th:convergence_moderate}. This requires some important intermediate steps, such as establishing the mild SPDE satisfied by $u^N$, see Subsection~\ref{subsec:eqempmeas}, and obtaining a bound on the moments of the mass of the system, see Subsection~\ref{subsec:boundparticles}. A crucial step consists in proving rates of convergence for the stochastic convolution integrals that appear in the SPDE satisfied by $u^N$, and this is done in  Subsection~\ref{subsec:stocint}. Then in Subsection~\ref{subsec:apriorimollempmeas}, we establish an \emph{a priori} bound on the moments of $u^N$, which will be used to carry out the proof of Theorem~\ref{th:convergence_moderate} in Subsection~\ref{subsec:proofThm2}. Finally we prove Corollary~\ref{cor:KR-convergence} in Subsection~\ref{subsec:proofCor}.
	
	
	\medskip
	
	\begin{acknowledgement}
	This work was partially supported by the SDAIM project ANR-22-CE40-0015, jointly funded by the S\~ao Paulo Research Foundation (FAPESP) and the French National Research Agency (ANR). TC acknowledges the support of the Fondation de Mathématiques Jacques Hadamard through the Hadamard lecturer program. AR acknowledges partial support from the MATH-AmSud project 24-MATH-04 Explore SDE.
	\end{acknowledgement}

	\section{Preliminaries and main results}
	\label{sec:prelim_and_main}

In this section, we first introduce some general notations that will be used throughout the paper.  We then state Theorem~\ref{Thm_WP_KS}, which is the main result concerning the well-posedness of the PDE \eqref{eq:KS_EDP} in an extended regime (compared to the usual Keller--Segel PDE). Then the interacting particle system and the branching mechanism are presented, leading to the definition of the empirical measure $\mu^N$ and the mollified empirical measure $u^N$. Once these objects are clearly identified, the main result of this paper, Theorem~\ref{th:convergence_moderate}, is stated.

	\subsection{Notations and standard properties}
	
\begin{itemize}[leftmargin=*]

\item For a Polish space $(X,d_X)$, denote $\mathcal{C}([0,T];X)$ the space of continuous functions from $[0,T]$ to $X$ endowed with the supremum norm.

\item  For functions from $\R^d$ to $\R$, denote the H\"older seminorm of parameter $\delta\in(0,1]$ by
\begin{equation}\label{eq:HolderNorm}
[f]_{\delta} \coloneqq \sup_{x\neq y \in \R^d} \frac{|f(x) - f(y)|}{|x-y|^\delta}.
\end{equation}
The space of continuous and bounded functions on $\R^d$ which have finite $[\cdot]_{\delta}$ seminorm is the H\"older space $\mathcal{C}^\delta(\R^d)$. For $\delta=1$, we will write $ \lVert f\rVert_{\text{Lip}}$ by a slight abuse of notation.

\item We will denote $L^p(\R^d)$ the usual Lebesgue space, or simply $L^p$, for any $p\in [1,\infty]$; the associated norm will be denoted $ \lVert f \rVert_{L^p}$. For the $L^m(\Omega)$ space, we will write $L^m_{\Omega}$ and denote the norm by $ \lVert f \rVert_{L^m_{\Omega}}$.

\item Denote by $\mathcal{P}(\R^d)$ the set of Borel probability measures on $\R^d$, and by $\mathcal{M}(\R^d)$ the space of finite nonnegative measures. 
Following \cite[Section 8.3]{BogachevII}, let us introduce the Kantorovich-Rubinstein metric which reads, for any two measures $\mu$ and $\nu$ on $\R^d$, 
\begin{equation}\label{eq:defKantorovich}
\|\mu - \nu \|_{0} = \sup \left\{ \int_{\R^d} \phi \, d(\mu-\nu) \, ; ~ \phi \text{ Lipschitz  with } \|\phi\|_{L^\infty}\leq 1 \text{ and } \|\phi\|_{\text{Lip}} \leq 1 \right\} ;
\end{equation}
this distance metrises the weak topology on $\mathcal{P}(\R^d)$ and $\mathcal{M}(\R^d)$, see \cite[Theorem 8.3.2]{BogachevII}.
 
 \item $(e^{t\Delta })_{t\geq 0 }$ denotes the semigroup of the heat operator on $\R^d$, \emph{i.e.} for $f \in {L}^p(\R^d)$, 
\begin{equation*}
e^{t \Delta}f (x)  
 = g_{2t} \ast f(x),
\end{equation*}
where $g_{t}(x) = \frac{1}{(2\pi t)^{d/2}} e^{-\frac{|x|^{2}}{2t}}$ is the usual heat kernel on $\R^d$.


\item For $\beta\in \R$ and $p\geq1$, denote $H^{\beta}_{p}(\mathbb{R}^{d})$ the Bessel potential space 
\begin{equation*}
H_{p}^{\beta}(\mathbb{R}^{d}) \coloneqq \Big\{ u \text{ tempered distribution; }  \, \mathcal{F}^{-1}\Big( \big(1+|\cdot|^{2}\big)^{\frac{\beta}{2}}\; \mathcal{F} u(\cdot) \Big) \in  L^p(\mathbb R^d)\Big\},
\end{equation*}
where $\mathcal Fu$ denotes the \emph{Fourier transform} of $u$. This space is endowed with the norm
 \begin{equation}\label{eq:Besselnorm}
 \| u \|_{\beta,p} = \Big\| \mathcal{F}^{-1}\big((1+|\cdot|^{2})^{\frac{\beta}{2}
}\; \mathcal{F} u(\cdot) \big) \Big\|_{L^p},
 \end{equation}
which is the $L^p$ norm associated to the Bessel potential operator $(I-\Delta)^\frac{\beta}{2}$ defined by
\begin{align*}
(I-\Delta)^\frac{\beta}{2} f \coloneqq \mathcal{F}^{-1}\left((1+|\cdot|^2)^{\frac{\beta}{2}} \mathcal{F}f \right) ,
\end{align*}
see \cite[p.180]{Triebel}.

\item When $u$ is a function or stochastic process defined on $[0,T]\times \R^d$, we will  denote $u_{t}$ the mapping  $x\mapsto u(t,x)$.

\item $a\wedge b$ (resp. $a\vee b$) denotes the minimum (resp. the maximum) between $a,b \in \R$. Let us also denote $a^+ = a \vee 0$ and $a^- = (-a)\vee 0$.

\item $\triangle$ will denote a cemetery state in which the particles are placed when they die.

\item $C$ refers to a constant whose value may change from line to line. To keep track of the dependence of a constant in certain parameters $p_{1}, \dots, p_{k}$, we will denote $C_{p_{1},\dots,p_{k}}$. 
We may also use the symbol $\lesssim$ for inequalities that involve multiplicative constants whose precise value is not relevant to our computations.

\end{itemize}

With the previous notations in hand, we recall some standard properties:
\begin{itemize}[leftmargin=*]
\item for any $n\in \N$ and $p\in [1,+\infty]$:
\begin{align}\label{eq:heat-estimate}
\lVert \nabla^n g_{t} \rVert_{L^p} \lesssim  t^{-\frac{1}{2}(n+d(1-\frac{1}{p}))} , \quad t>0.
\end{align}

\item By interpolation (see \emph{e.g.} \cite[Eq. (2) p.87]{RunstSickel}), it follows that for any $n\in \N$, any $\gamma\geq 0$ and any $p\in (1,+\infty)$,
\begin{align*}
\lVert (I-\Delta)^\frac{\gamma}{2} \nabla^n g_{t} \rVert_{L^p} \lesssim  (1\wedge t)^{-\frac{1}{2}(\gamma+n+d(1-\frac{1}{p}))} , \quad t>0.
\end{align*}

\item As a consequence of the previous inequality and a convolution inequality, it follows that for any $1\leq r\leq p$ and any measurable $f$,
\begin{align}\label{eq:Besov-heat-estimate}
\lVert (I-\Delta)^\frac{\gamma}{2} \nabla^n e^{t\Delta} f \rVert_{L^p} \lesssim  (1\wedge t)^{-\frac{1}{2}(\gamma+n+d(\frac{1}{r}-\frac{1}{p}))} \lVert f \rVert_{L^r} , \quad t>0.
\end{align}

\item The following embeddings between Bessel spaces hold (see \emph{e.g.} \cite[Theorem p.31]{RunstSickel}): for any $\beta\leq \beta'\in \R$, any $q\geq p\geq1$,
\begin{align}\label{eq:Bessel_embedding}
H^{\beta'}_{p}(\R^d) \hookrightarrow H^\beta_{p}(\R^d) \hookrightarrow H^{\beta-d(\frac{1}{p}-\frac{1}{q})}_{q}(\R^d) .
\end{align}
In addition, for $\gamma\in (0,1)$ and $q=+\infty$,
\begin{equation}\label{eq:embedHolder}
H^\gamma_{q}(\R^d) \hookrightarrow \mathcal{C}^\gamma(\R^d),
\end{equation}
see \emph{e.g.} \cite[Eq. (9) p.32]{RunstSickel}.
\end{itemize}

	\subsection{Well-posedness of the logistic Keller--Segel PDE}
	\label{subsec:WP_PDE}

The initial condition $u_0$ is supposed to belong to $L^1\cap L^{\infty}(\R^d)$ and to be non-negative with total mass $\int_{\R^d} u_0(x) \, \dd x =1$. Note that with the two non-vanishing constants $\mu$ and $\nu$, the mass of the system is not constant over time.

The PDE \eqref{eq:KS_EDP} has been studied in \cite{TelloWinkler} and in several subsequent works, mostly on a bounded space domain and with Neumann boundary conditions. We aim here to extend this well-posedness to the whole space for mild solutions, assuming only bounded initial conditions.

\begin{definition}[Local mild solution]
	Let $T>0$. A measurable function $u$ defined on $[0,T]\times \R^d$ is said to be a mild solution of \eqref{eq:KS_EDP} if it satisfies the following properties:\begin{enumerate}[label=(\roman*)]
		\item $u \in \CC([0,T];L^1) \cap L^\infty([0,T];L^{\infty})$; 
		\item For any $t \in [0,T]$,
		 \begin{equation}\label{eq:KSmild}
			u_t = e^{t\Delta}u_0 - \chi \int_0^t \nabla \cdot e^{(t-s)\Delta} (u_s\, K\ast u_s) \, \dd s + \int_0^t e^{(t-s)\Delta} (\nu u_s - \mu u_s^2) \, \dd s.
		\end{equation}
	\end{enumerate}
\end{definition}

\begin{definition}[Global mild solution]
A function $u$ defined on $\R_+ \times \R^d$ is said to be a global mild solution to \eqref{eq:KS_EDP} if it is a mild solution on $[0,T]\times \R^d$ for any $T>0$. 
\end{definition}

Our first main result starts with a local well-posedness statement for mild solutions to \eqref{eq:KS_EDP}, followed by a global well-posedness statement which represents the suppression of blow-up phenomenon due to the logistic term when $\mu$ is large enough.

\begin{theorem}\label{Thm_WP_KS}
	Let $d\geq2$ and $\mu,\nu \geq 0$. Assume that $u_{0}\in L^1\cap L^\infty$. 
	\begin{enumerate}[label=(\alph*)]
	\item There exists a maximal existence time $T^*\in (0,+\infty]$ such that for any $T\in (0,T^*)$, the PDE~\eqref{eq:KS_EDP} admits a unique mild solution on $[0,T]$.
	
	\item Assume further that 
	$$\mu > \frac{d-2}{d} \chi.$$
	 Then there exists a unique global mild solution to \eqref{eq:KS_EDP}, \emph{i.e.} $T^*=+\infty$. Moreover, the solution $u$ satisfies 
		\begin{equation*}
			\forall T >0, \quad \sup_{t \in[0,T]} \Vert u_t \Vert_{L^1\cap L^\infty} < + \infty.
		\end{equation*}
	\end{enumerate}
\end{theorem}

The particular case of dimension $2$ is interesting, since the previous theorem states that no blow-up can occur, whatever the strength of the chemo-attractivity parameter $\chi$.

\begin{remark}
The maximality of $T^*$ will be described in Proposition~\ref{prop:explosion-criterion}. It simply means that either $T^* = +\infty$ and the solution is global, or $T^* <\infty$ and in this case the norm of the solution has to diverge as $T$ approaches $T^*$. See \eqref{KS_explosion_criterion} for the precise statement.
\end{remark}



\subsection{Definition of the branching interacting particle system}
\label{subsubsec:constructionIPS}

We now want to approximate the mild solution to \eqref{eq:KS_EDP} by a moderately interacting particle system with births and deaths. Let $N \geq 1$ denote the initial number of particles and $(\xi^{k_{0}})_{k_{0} \in \N^*}$ be a sequence of identically distributed (not necessarily independent) random variables with common law $u_0$, which give the initial positions of the particles. 

During its lifetime, a particle can give birth to two new particles at its current location and disappear, or it can just die and is thus removed from the system. To follow the genealogy of particles, we adopt the so-called Ulam--Harris--Neveu notation for the space of labels of particles:
\begin{equation*}
\Lambda^N \coloneqq \bigcup_{n \in \N} \{1,\dots,N\}\times \{1,2\}^n = \big\{ (k_0,i_1,\dots, i_n), \, i_1,\dots i_n \in \{1,2\}, \, k_0 \in \{1,\dots, N\}, \, n \in \N \big\}.
\end{equation*}

For a particle $k = (k_0,i_1,\dots,i_n) \in \Lambda^N$ with $n \geq 1$, we denote by $(k,-)\coloneqq (k_0,i_1,\dots,i_{n-1})$ its mother. Each particle $k \in \Lambda^N$ which appears in the system has a (almost surely) finite lifetime $L^k = [T^k_0,T^k_1)$, where $T^k_0$ is the stopping time at which the particle appears ($T^k_0$ can be $0$ for the $N$ initial particles) and $T^k_1$ is the stopping time at which it disappears from the system. For now we sketch the mechanism of the demographic events, their precise construction will be given at the end of this section:
\begin{itemize}
\item After some time in the system, a particle $k\in \Lambda^N$ will eventually disappear, either by giving birth to $2$ new particles, or simply by  being killed due to the local competition with other particles. We denote by $T^{k,\text{div}}_1$ the possible time at which $k$ divides and by $T^{k,\text{die}}_1$ the time at which it possibly dies, so that $T^k_1= T^{k,\text{div}}_1  \wedge T^{k,\text{die}}_1$;  denote also by $I_t \subset \Lambda^N$ the set of particles alive at time $t$.

\item \emph{The birth mechanism}: A particle $k \in \Lambda^N$ appears when its mother $(k,-)$ divides, which happens at random exponential rate $\nu$ if the mother has not died before, \emph{i.e.} if $T^{(k,-),\text{div}}_{1}< T^{(k,-),\text{die}}_{1}.$ In this case, we set $T^k_0 = T^{(k,-)}_1$ and $X^k_{T^k_0}\coloneqq \lim_{t \nearrow T^{(k,-)}_1 } X^{(k,-)}_{t} \eqqcolon X^{(k,-)}_{T^k_1}$. 

\item \emph{The death mechanism}: As for its disappearance from the system, it can result from two different type of events, either the particle divides giving birth to two new particles as described above, or it is killed. The latter case can happen due to local competition with other particles, and in this case death happens at an exponential rate proportional to the local density of alive particles; the precise mechanism will be described right after the density of particles will have been introduced in \eqref{eq:density-particles}. 
After a particle dies or divides, we set its value to some ground state $\triangle \notin \R$. If a particle $k \in \Lambda^N$ never appears, for example because its mother died, we set $T^k_0 = T^k_1 = + \infty$ and $X^k_t = \triangle$ for any $t$.
\end{itemize}

 \medskip
 
  Let us now construct the particle system $(X^k_t)_{k \in \Lambda^N, \, t \geq 0}$ between the demographic events. At time $t=0$, the $N$ initial particles satisfy $X^k_{0} = \xi^{k_{0}}$ for $k = (k_{0})$, $k_{0}\in \{1,\dots, N\}$.

%
%
%
Let $\theta \in \CC^{\infty}(\R^d;\R_+)$ with compact support included in the unit ball $B(0,1)$ and such that ${\int_{\R^d} \theta(x)\, \dd x =1}$. Then, for $\alpha>0$, define the mollifying sequence $(\theta^N)_{N\in \N^*}$ as follows:
\begin{equation}\label{eq:def-thetaN}
\theta^N : x \in \R^d \mapsto N^{\alpha d}\theta (N^{\alpha}x).
\end{equation}
	For $A>0$, we will use the smooth cutoff function $F_{A}:\R^d \to \R^d$ defined component-wise for each $k\in \{1,\dots, d\}$ by
	\begin{equation*}
	(F_{A})_{k}(x_{k}) 
	= \left\{
	\begin{array}{ll}
	x_{k} & ~~ \text{ if }~ x_{k}\in [-A,A],\\
	A & ~~ \text{ if }~ x_{k}>A+1,\\
	-A & ~~ \text{ if }~ x_{k}<-A-1,
	\end{array}
	\right.
	\end{equation*}
	and such that $ \lVert (F_{A})_{k}' \rVert_{\infty}\leq 1$, $ \lVert (F_{A})_{k}'' \rVert_{\infty} <\infty$; hence in particular, $ \lVert F_{A} \rVert_{\infty}\leq A+1$.

	Each particle $k \in \Lambda^N$ has the following dynamics during its lifetime:
\begin{equation}\label{eq:IPS}
\dd X^{k}_{t} = \chi F_A \bigg(\frac1N \sum_{j \in I_t} K\ast \theta^N (X^{k}_t - X^{j}_t)\bigg)\, \dd t + \sqrt{2}\, \dd B^k_t, \quad t \in [T^k_0, T^k_1),
\end{equation}
where $(B^k)_{k \in \Lambda^N, N\in \N^*}$ is an i.i.d.\ sequence of standard Brownian motions in $\R^d$ on some filtered probability space $(\Omega,\mathcal{F}, (\mathcal{F}_{t})_{t\geq 0}, \PP)$, independent from the initial positions $(\xi^{k_{0}})_{k_{0}\in \N^*}$ of the particles, and also independent from the demographic events (see further). 

\medskip

We are interested in the convergence of the empirical measure of the system
\begin{equation*}
\mu^N_t \coloneqq \frac{1}{N} \sum_{k\in I_{t}} \delta_{X^k_{t}}
\end{equation*}
and of the mollified empirical measure
\begin{equation}\label{eq:density-particles}
u^N_t \coloneqq \theta^N \ast \mu^N_t = \frac{1}{N} \sum_{k\in I_{t}} \theta^N(\cdot-X^k_{t})
\end{equation}
towards $u_t$, the solution of \eqref{eq:KS_EDP}.

\medskip

Let us finally define precisely the demographic events. 
On the same filtered probability space $(\Omega,\mathcal{F}, (\mathcal{F}_{t})_{t\geq 0}, \PP)$, let $\NN$ be a Poisson random measure on $\R^+ \times \Lambda^N \times \R^+$ with intensity measure $\lambda \otimes m \otimes \lambda$, where $m$ is the counting measure on $\Lambda^N$ and $\lambda$ denotes the usual Lebesgue measure; $\NN$ is further assumed to be independent from $(\xi^{k_{0}})_{k_{0}\in \N^*}$ and $(B^k)_{k \in \Lambda^N, N\in \N^*}$. Then let
\begin{equation}\label{eq:defT1}
\begin{split}
& T^{k,\text{div}}_1 \coloneqq \inf \big\{ t> T_{0}^k:~ \NN\big( [T_{0}^k,t]\times \{k\}\times [0,\nu] \big) \neq 0 \big\}\\
& T^{k,\text{die}}_1 \coloneqq \inf \Big\{ t> T_{0}^k:~ \int_{[T_{0}^k, t] \times \R_{+}} \1_{\big\{\nu\leq \rho < \nu + \mu u^N_{s^-}(X^k_{s^-})\wedge A\big\}}\, \NN (\dd s, \{k\},\dd \rho) \neq 0 \Big\} ,
\end{split}
\end{equation}
%
where we choose $A$ to be the same parameter as in the definition of $F_{A}$.
Finally recall that 
\begin{equation*}
T^k_1 = T^{k,\text{div}}_1 \wedge T^{k,\text{die}}_1.
\end{equation*}

\subsection{Quantitative approximation of the PDE by the particle system}
\label{subsec:ThmConv}

The assumptions on the parameters of the particle systems are of two different natures. First it concerns the regularisation parameter $\alpha$ introduced in \eqref{eq:def-thetaN}.
\begin{Assumption}\label{Assumptions_parameters}
	Let $r \in (d,+\infty)$ and $\gamma\in [\frac{d}{r},1)$. 
	Assume that 
	\begin{equation}\label{Eq:contrainte_beta} 
	0 < \alpha <\frac{1}{2(d +\gamma-\frac{d}{r})} .
	\end{equation} 
\end{Assumption}

Second, it concerns the regularity of the initial law of the particles.
\begin{Assumption}\label{Assumptions_Init}
Let $r \in (d,+\infty)$, $\gamma\in [\frac{d}{r},1)$ and $\alpha$ as in Assumption~\ref{Assumptions_parameters}. For $\mu_{0}^N = \frac{1}{N} \sum_{k=1}^N \delta_{\xi^k}$ and $u_{0}^N = \theta^N\ast \mu^N_{0}$, 
assume that for any $m \geq 1$,
	\begin{equation*}
	\sup_{N\geq 1}  \E \Vert u^N_0 \Vert_{\gamma_{0}, r}^m  <+ \infty \quad \text{for}~ \gamma_{0}\in [0,1)~ \text{such that}~~  \gamma + (\gamma-\gamma_{0})\vee 0 <1 \, ;
	\end{equation*}
	 and assume that for any $\kappa \geq 2$ and $m \geq 1$,
	 $$\sup_{N \geq 1}\E \Big[\langle \mu^N_0, |\cdot|^{\kappa} \rangle^m\Big]  < +\infty.$$ 
\end{Assumption}

For a given time horizon $T>0$ in the PDE, let us define the cutoff threshold $A_T$ by 
\begin{equation}\label{Def_cutoff_limite} 
	A_T\coloneqq \sup_{t \in [0,T]} \Vert u_t \Vert_{L^{\infty}} + \sup_{t \in [0,T]} \Vert K\ast u_t \Vert_{L^\infty};
\end{equation}
we will see that, as a consequence of Proposition~\ref{KS_properties_kernel} and Theorem~\ref{Thm_WP_KS}, $A_{T}$ is finite for any $T<T^*$.

We are now ready to state the main result of the paper, which gives a rate of convergence of the mollified empirical measure towards the solution of the PDE~\eqref{eq:KSmild}.

\begin{theorem} \label{th:convergence_moderate}
Let $d\geq2$ and $\mu,\nu \geq 0$. Assume that $u_{0}\in L^1\cap L^\infty$. 
Recall that $T^*$ is the maximal existence time of the mild solution of \eqref{eq:KS_EDP} given by Theorem~\ref{Thm_WP_KS}. 
Let $T \in (0,T^*)$ and let $A\geq A_T$, for $A_{T}$ as defined in \eqref{Def_cutoff_limite}. Let the Assumptions \ref{Assumptions_parameters} and \ref{Assumptions_Init} hold. Then for any $\eps > 0$ and any $m \geq 1$, there exists a constant $C > 0$ such that for any $ N \in \N^*$,
\begin{equation}\label{eq:quantitative_convergence_KS}
	\sup_{t \in [0,T]} \big\Vert \Vert u^N_t - u_t \Vert_{L^1 \cap L^r} \big\Vert_{L^m_\Omega} \leq C \Big( \big\Vert \Vert u^N_0 - u_0 \Vert_{L^1 \cap L^r} \big\Vert_{L^m_\Omega} + N^{-\varrho +\eps}\Big), 
\end{equation}
where 
\begin{equation}\label{rate_of_cv_thm}
\varrho \coloneqq \min \left\{ \alpha \big(\gamma-\frac{d}{r}\big),\, \frac12 \Big( 1 - 2\alpha d\big( 1 - \frac1r \big) \Big) \right\} .
\end{equation}
\end{theorem}

For  two finite measures $\mu,\nu \in \MM(\R^d)$, recall that $ \lVert\mu-\nu\rVert_{0}$ denotes the Kantorovich-Rubinstein distance defined in \eqref{eq:defKantorovich}. We have the following corollary, which gives the convergence of the genuine (\emph{i.e.} non-mollified) empirical measure.
\begin{corollary}\label{cor:KR-convergence}
	Let us assume that the same assumptions as in Theorem~\ref{th:convergence_moderate} hold and let $\varrho$ be given by \eqref{rate_of_cv_thm}. Then for any $\eps \in (0 ,\varrho)$, there exists $C>0$ such that for any $N\geq 1$,
	\begin{equation*}
		\sup_{t \in [0,T]} \big\Vert \Vert \mu^N_t - u_t \Vert_0 \big\Vert_{L^m_\Omega} \leq C \Big( \big\Vert \Vert u^N_0 - u_0 \Vert_{L^1\cap L^r} \big\Vert_{L^m_\Omega} + N^{-\varrho +\eps}\Big).
	\end{equation*}
\end{corollary}

\begin{remark}[I.i.d. initial conditions]
In Theorem~\ref{th:convergence_moderate} and Corollary~\ref{cor:KR-convergence}, the particles do not need to be initially independent and identically distributed. We refer the reader to the Appendix~B of \cite{ORT-Burgers}, where a rate of convergence is obtained for $\big\Vert \Vert u^N_0 - u_0 \Vert_{L^1\cap L^r} \big\Vert_{L^m_\Omega}$ assuming that the particles are initially i.i.d. with law $u_{0} \in H^\beta_{r}(\R^d)$, for some $\beta>0$. In \cite{HJMRZ}, a similar statement holds in total variation without assuming additional regularity on $u_{0}$.
\end{remark}


\begin{remark}[Heuristics on the rate $\varrho$]
In view of Assumption~\ref{Assumptions_parameters}, we have $\varrho\geq 0$ since $\alpha>0$, $\gamma\geq \frac{d}{r}$ and 
\begin{equation*}
1 - 2\alpha d\big( 1 - \frac1r \big) \geq 1 - 2\alpha \big( d +\gamma - \frac dr \big) > 0.
\end{equation*}
%
We see in \eqref{rate_of_cv_thm} that the rate $\varrho$ is the minimum between two contributions. In fact, the error originates from three different terms:
\begin{enumerate}[label=(\arabic*)]
\itemsep0em 
\item\label{it:contrib1} The first contribution is related to the H\"older regularity of $u^N$, uniformly in $N$, or equivalently to the H\"older regularity of the PDE $u$; we can prove that $[u_{t}]_{\gamma-\frac{d}{r}} \in L^1([0,T])$. This regularity can be leveraged to control the rate of convergence of the death term $-\mu (u^N_{t})^2$ towards $-\mu (u_{t})^2$, see more precisely \eqref{eq:rateu2}, to give the $N^{-\alpha(\gamma-\frac{d}{r})}$ contribution.

\item Similarly, the second contribution comes from the H\"older regularity of $K\ast u_{t}$, which is $1-\frac{d}{r}$, see more precisely \eqref{eq:rateKu}. Since this value is larger than $\gamma-\frac{d}{r}$ obtained in \ref{it:contrib1}, it does not appear in the final result.

\item The third contribution comes from the two stochastic integrals that appear in the Stochastic PDE satisfied by $u^N$. Using an infinite-dimensional version of the BDG inequality in $L^1\cap L^r$, for both Brownian integrals (Proposition~\ref{KS_estimate_Lp_Brownian_integral}) and Poisson integrals (Proposition~\ref{KS_estimate_Lp_Poisson_integral}), we obtain the rate $\frac12 \big( 1 - 2\alpha d\big( 1 - \frac1r \big) \big)$ which is optimal for such stochastic integrals in this $L^1\cap L^r$ norm. 

\end{enumerate}
\end{remark}

\begin{remark}[Optimising $\alpha$ or $\varrho$]
\label{rk:rate}
Playing with the parameters from Assumptions~\ref{Assumptions_parameters} and \ref{Assumptions_Init}, one can try to maximise $\alpha$ to make the particles look more similar to `physical' particles (\emph{i.e.} without smoothing); or try to get $\varrho$ as large as possible to maximise the convergence rate.
	\begin{itemize}
	\itemsep0em 
	\item First,  to make $\alpha$ as large as possible, Assumption~\ref{Assumptions_parameters} shows that one has to take $\gamma-\frac{d}{r}$ close to $0$. In this case, the restriction is
	 \begin{equation*}
	 0 < \alpha < \frac{1}{2d}.
	 \end{equation*}
	Choosing $\alpha$ close to $\frac{1}{2d}$ imposes that the rate of convergence $\varrho$ in Theorem~\ref{th:convergence_moderate} tends to $0$. 
	
	\item Second, to maximise $\varrho$, we assume that $\gamma_{0}$ in Assumption~\ref{Assumptions_Init} can be chosen close to $1$, so that one can also pick ${\gamma\sim1}$. Hence we now want to maximise 
	$$\min\left(\alpha \left(1 - \frac{d}{r} \right), \frac12 \left(1 - 2\alpha d \left(1 - \frac1r\right)\right)\right).$$
	 One can easily check that for $r\sim+\infty$ and $\alpha = \frac{1}{2(d+1)}$, $\varrho$ is maximised and we then have $\varrho = \frac{1}{2(d+1)}- \varepsilon$ for $\varepsilon>0$ small.
	 This is the same maximal rate obtained for the standard Keller--Segel equation in \cite[Section $5.2$]{Pisa}.
	\end{itemize}
\end{remark}

\begin{remark}[Removing the cutoff in \eqref{eq:IPS}]
In a similar framework but with particles living in the torus, an almost sure  convergence was obtained in \cite{ORT-Burgers} for Keller--Segel particles without cutoff. A similar almost sure result would probably hold here too. However in the present setup of $L^m_{\Omega}$ convergence, we leave the question of removing the cutoff for future research, with the aim of controlling the agglomeration of particles using the death mechanism when particles get too close to one another.

We also mention that in \cite{HJMRZ}, the authors work with no cutoff on the particles and manage to prove that $\sup_{N} \sup_{t\in [0,T]} \E[|X^{i,N}_{t}|^m]<\infty$, which is crucial to bound the norm of the stochastic convolution integrals. However this approach, which works for kernels with sub-Coulombic singularity, might not be directly adapted here.
\end{remark}


\section{\emph{A priori} estimates and well-posedness of the PDE}
\label{sec:PDE}

In this section, we prove Theorem~\ref{Thm_WP_KS}. First, some properties of the Keller--Segel kernel are stated. Then, we prove the local well-posedness of Equation~\eqref{eq:KS_EDP} in mild form (see~\eqref{eq:KSmild}), using a contraction principle, and establishing a blow-up criterion for the mild equation. Then an adaptation of the argument from \cite{TelloWinkler}, presented in Proposition~\ref{prop:estimates-maxsol}, gives a condition on $\mu$, $\chi$ and $d$ to ensure the boundedness of the norm of the solution of \eqref{eq:KSmild}. As in \cite{TelloWinkler}, we could have proven Theorem~\ref{Thm_WP_KS} with a more general logistic term $g(u)$ in the PDE~\eqref{eq:KS_EDP}, instead of $\nu u - \mu u^2$; however, since our probabilistic approximation of this term corresponds to a precise interpretation in terms of a constant birth rate $\nu$ for the particles and death rate $\mu u$, we chose to work with this less general form.

\subsection{Properties of the Keller--Segel interaction kernel}

Let us give some regularising properties of the interaction kernel $K$, defined in \eqref{eq:defK}. In particular, we see below that the convolution operator with kernel $K$ transfers integrability to H\"older regularity, which will be useful to obtain a rate of convergence in Theorem~\ref{th:convergence_moderate}.

\begin{proposition}\label{KS_properties_kernel}
	The Keller--Segel kernel $K$ satisfies the following properties:
	\begin{enumerate}[label=(\roman*)]
		
		\item\label{item:infboundK} Let $ p \in [1,d')$ and $q \in (d',+\infty]$, where $d'$ (resp. $p', q'$) denotes the H\"older conjugate of $d$ (resp. $p,q$). There exists $C_{K,d} \equiv C_{K,d,p,q}>0$ such that for any $f \in L^{p'} \cap L^{q'}$, 
		\begin{equation}\label{KS_properties_kernel_L_infty}
			\Vert K\ast f \Vert_{L^{\infty}} \leq C_{K,d} \Vert f \Vert_{L^{p'} \cap L^{q'}}.
		\end{equation}
		
		\item\label{item:HolderboundK} For any $r \in (d, +\infty)$, let $\zeta = 1 - \frac{d}{r}$. Then, there exists $C>0$ such that for any $f \in L^{1}\cap L^r$, 
		$$ [K\ast f]_{\zeta} \leq C \Vert f \Vert_{L^{1} \cap L^r}.$$
	\end{enumerate}
\end{proposition}

\begin{proof}
First, an elementary observation yields that for $B(0,1)$ the unit ball of $\R^d$,
\begin{equation}\label{eq:integrabilityK}
K \in L^p(B(0,1)) \quad \text{and} \quad K \in L^q(B(0,1)^c).
\end{equation}
The proof of \ref{item:infboundK} follows from decomposing the convolution $K\ast f$ as a sum on the unit ball and on its complementary, then applying H\"older's inequality with the exponents identified in \eqref{eq:integrabilityK}.

For \ref{item:HolderboundK}, let us observe first that $\nabla K$ defines a convolution operator of Calder\'on-Zygmund type, in the sense that it satisfies the size, smoothness and cancellation conditions, see Equations (4.4.1)-(4.4.3) in \cite{Grafakos}. Hence it yields a convolution operator that is bounded in any $L^z$, see \cite[Theorem~4.4.1]{Grafakos}:
\begin{align}\label{eq:boundednessGradK}
 \big\|  \nabla K\ast f \big\|_{L^r} \lesssim \|f\|_{L^r} <\infty.
\end{align}
Thus \ref{item:HolderboundK} follows from \eqref{eq:boundednessGradK} and a minor adaptation of \cite[Lemma 5.1]{Pisa}.
\end{proof}

\begin{remark}
Let $r\in (d,+\infty]$. Proposition~\ref{KS_properties_kernel} implies that for any $f \in L^1\cap L^r$, 
		\begin{equation}\label{Bound_K_convol_L_infty}
		\Vert K\ast f\Vert_{L^{\infty}}\leq C_{K,d} \Vert f\Vert_{L^1\cap L^r}.
		\end{equation}
\end{remark}

\subsection{Global well-posedness of the PDE: Proof of Theorem~\ref{Thm_WP_KS}}

The proof of Theorem~\ref{Thm_WP_KS} results from the combination of the two following propositions. 

First, we state the local well-posedness for the mild solution of the Keller--Segel equation~\eqref{eq:KS_EDP}; then in line with classical criteria to determine whether a differential equation has global solutions, we formulate a specific criterion in the precise space in which we solve~\eqref{eq:KSmild}.
\begin{proposition}\label{prop:explosion-criterion}
	Let $u_{0}\in L^1\cap L^\infty$. There exist $T>0$ and a unique mild solution $u$ to \eqref{eq:KSmild} on $[0,T]$. In particular, we recall that this entails 
	\begin{equation*}
	u\in \CC([0,T];L^1) \cap L^\infty([0,T];L^{\infty}).
	\end{equation*}
	We can thus construct the unique maximal solution $u$ to \eqref{eq:KSmild} and we denote by $T^*$ its existence time. Moreover, either $T^* = +\infty$ and the solution is global, or $T^* <\infty$ and in this case 
	\begin{equation}\label{KS_explosion_criterion}
	\sup_{t \in[0,T]} \Vert u_t \Vert_{L^1} + \sup_{t \in [0,T]} \Vert u_t \Vert_{L^{\infty}} \underset{T \to T^*}{\longrightarrow} + \infty.
	\end{equation}
\end{proposition}

Second, we will prove that in the regime where the local death rate of particles is sufficiently strong, the maximal solution of \eqref{eq:KSmild} is bounded on any time interval.
\begin{proposition}\label{prop:estimates-maxsol}
Let $\mu > \frac{d-2}{d} \chi$. We have the following estimates on the maximal mild solution $u$ to \eqref{eq:KS_EDP}. 
\begin{enumerate}[label=(\roman*)]
	\item For any $t \in [0,T^*)$, $u_t$ is non-negative almost everywhere. 
	
	\item\label{item:locbound} For any $p \in [1,+\infty]$, there exists a locally-bounded, non-decreasing function ${T \in \R_+ \mapsto F(p,T)} \in \R_+^*$ such that for any $T < T^*$, 
\begin{equation}\label{KS_a_priori_bound_L_gamma}
\sup_{t\in [0,T]} \Vert u_t \Vert_{L^p} \leq F(p,T).
\end{equation}

\end{enumerate}	
\end{proposition}

\begin{proof}[Proof of Theorem~\ref{Thm_WP_KS}]
In view of Proposition~\ref{prop:estimates-maxsol}\ref{item:locbound}, it is now clear that in the regime $\mu > \frac{d-2}{d} \chi$, the blow-up described in Equation~\eqref{KS_explosion_criterion} cannot happen. Thus in this regime, the Keller--Segel equation admits a unique global solution.
\end{proof}

\medskip

In the rest of this section, we prove successively Proposition~\ref{prop:explosion-criterion} and Proposition~\ref{prop:estimates-maxsol}.

\begin{proof}[Proof of Proposition~\ref{prop:explosion-criterion}]
	Let $T>0$ to be chosen later and denote $E_T\coloneqq \CC([0,T];L^1) \cap L^\infty([0,T];L^{\infty})$ which is a Banach space with respect to the norm $\Vert \cdot \Vert_{T}$ defined, for $f \in E_T$, by
	\begin{equation*}
	\Vert f \Vert_T \coloneqq \sup_{t \in[0,T]} \Vert f_t \Vert_{L^1} + \sup_{t \in [0,T]} \Vert f_t \Vert_{L^{\infty}}.
	\end{equation*}
	 The solution of the mild equation \eqref{eq:KSmild} can be rewritten as $u = e^{\cdot \Delta }u_0 - B(u,u) +  L(u)$, where the bilinear functional $B$ is given by
	  \begin{equation*}
	  B : (u,v) \in E_T \times E_T \mapsto \left( t \in [0,T] \mapsto  \chi \int_0^t \nabla \cdot e^{(t-s)\Delta} ((K\ast u_s)v_s) \, \dd s + \mu \int_0^t e^{(t-s)\Delta}  (u_s v_{s}) \, \dd s \right)
	  \end{equation*}
	  and the linear functional $L$ is given by
	  \begin{equation*}
	  L : u \in E_T \mapsto \left( t \in [0,T] \mapsto   \nu\int_0^t e^{(t-s)\Delta} u_s \, \dd s \right).
	  \end{equation*}
	  Let us prove the continuity of $B$ and $L$ in $E_{T}$. For $ u \in E_T$, using \eqref{eq:heat-estimate} with $n=0$, it comes
	  \begin{equation*}
	  \Vert L(u)_{t}\Vert_{L^1} \leq C t \sup_{s\leq t} \Vert u_{s} \Vert_{L^1} ~~\text{and}~~ \Vert L(u)_{t}\Vert_{L^\infty} \leq C t \sup_{s\leq t} \Vert u_{s} \Vert_{L^\infty},
	  \end{equation*}
	  so that
	\begin{align}\label{KS_continuity_linear}
	 	\Vert L(u)\Vert_{T} \leq C T\Vert u \Vert_{T},
	 \end{align} 
 where $C$ does not depend on $T$. 

 For $u,v \in E_T$, using \eqref{eq:heat-estimate} with $n=0$ and $n=1$ and \eqref{Bound_K_convol_L_infty} with $r=+\infty$, we get similarly that there exists a positive constant $C$ independent of $T$ such that for $p\in \{1,+\infty\}$,
 \begin{align*}
 \Vert B(u,v)_{t}\Vert_{L^p} \leq C (\sqrt{t}+t)\, \sup_{s\leq t} \big(\Vert u_{s} \Vert_{L^1\cap L^\infty} \Vert v_{s} \Vert_{L^p}\big);
 \end{align*}
hence this yields 
 \begin{equation}\label{KS_continuity_bilin}
 \Vert B(u,v) \Vert_{T}  \leq C (\sqrt{T}+T) \Vert u\Vert_T \Vert v \Vert_T.
\end{equation} 
Denoting $F(u) \coloneqq e^{\cdot \Delta} u_{0} - B(u,u) +  L(u)$, it follows from \eqref{KS_continuity_linear} and \eqref{KS_continuity_bilin} that for any $u,v \in E_T$,
\begin{align}\label{KS_contraction_existence}
\Vert F(u) - F(v) \Vert_T &\leq \Vert L(u-v) \Vert_T +\Vert B(u,u-v) \Vert_T +\Vert B(u-v,v) \Vert_T  \nonumber\\
&\leq C \left((\sqrt{T}+T) (\Vert u \Vert_T + \Vert v \Vert_T)  +  T \right) \Vert u-v\Vert_T.
\end{align}

Define $R = \Vert u_0 \Vert_{L^1} + \Vert u_0 \Vert_{L^{\infty}} +1$. We show now that if $T$ is small enough, the closed ball $\overline{B}(0,R)$ is stable by $F$ and that $F$ is a contraction on $\overline{B}(0,R)$. Using \eqref{KS_contraction_existence}, we see that $ \lVert F(u) - F(v)\rVert_{T} \leq \lVert u-v\rVert_{T}$ if 
$$ C(\sqrt{T} +T)R^2 + C TR \leq 1.$$ 
This is the case for example for 
\begin{equation}\label{KS_time_local_existence} 
T \coloneqq \min \left(\frac{1}{(3CR^2)^2}, \frac{1}{3CR^2}, \frac{1}{3CR}\right).\end{equation}
Dividing $T$ by $2$ if necessary, we see that $F$ is a contraction on $\overline{B}(0,R)$. By a classical fixed-point argument, there exists a unique mild solution in $\overline{B}(0,R)$.

\medskip

Let us prove the local uniqueness. Let $u$ and $v$ be two mild solutions on $[0,T]$. Then, for $t$ small enough, the restrictions of $u$ and $v$ to $[0,t]$ belong to $\overline{B}(0,R)$ 
 and thus coincide, since $F$ is a contraction on $\overline{B}(0,R)$. Let us consider $I = \{ t \in [0,T]: \, u_{s} = v_{s} \text{ for all } s\leq t \},$ which is not empty. By initialising the mild formulation at $t \in I$, we see with the same argument as before that $I$ is open. The fact that $I$ is closed follows from the continuity in time of $u$ and $v$. By connectedness, $I = [0,T]$. Hence, gluing local solutions 
 and using the local uniqueness, we can build a unique maximal mild solution $u$ defined on the maximal time interval $[0,T^*)$.

\medskip

Let us now focus on the explosion criterion \eqref{KS_explosion_criterion}. Assume by contradiction that $T^*$ is finite and that $\sup_{t \in [0,T^*)} \Vert u_t \Vert_{L^1} +  \sup_{t \in [0,T^*)} \Vert u_t \Vert_{L^{\infty}} < + \infty$. Then, we see that the local existence time given in \eqref{KS_time_local_existence} for the mild equation \eqref{eq:KSmild} starting at time $t \in [0,T^*)$ by the solution $u_t$ can be lower-bounded uniformly in $t \in [0,T^*)$. We can thus build a solution defined on a time interval $[0,T]$ with $T> T^*$. This is a contradiction. 
\end{proof}

The last ingredient to complete the proof of Theorem~\ref{Thm_WP_KS} consists in the proof of Proposition~\ref{prop:estimates-maxsol}. This is an adaptation to the whole space  and to solutions with lower regularity of the proof proposed initially by Tello and Winkler~\cite{TelloWinkler}.

\begin{proof}[Proof of Proposition~\ref{prop:estimates-maxsol}]
	\textbf{Step 1: Non-negativity of the solution.} Let $T^*$ be the maximal existence time given by Proposition~\ref{prop:explosion-criterion}, and let $T < T^*$. We first prove that the mild solution $u$ on $[0,T]$ is a weak solution to \eqref{eq:KS_EDP}. For a fixed test function $\varphi \in \CC_c^{\infty}(\R^d)$, using Fubini's theorem and the integration-by-parts formula, one gets
	\begin{equation*}
	\langle u_t, \varphi \rangle = \langle u_0, e^{t\Delta}\varphi \rangle + \chi \int_0^t \langle  (K\ast u_s)u_s, e^{(t-s)\Delta}\nabla \varphi \rangle \, \dd s + \int_0^t \langle  \nu u_s - \mu u_s^2, e^{(t-s)\Delta} \varphi \rangle \, \dd s.	
	\end{equation*}
	Recall from Proposition~\ref{prop:explosion-criterion} that $u \in \CC([0,T];L^1) \cap L^\infty([0,T]; L^{\infty})$. Thus using that $\varphi \in \CC_c^{\infty}(\R^d)$, we deduce that for any $t \in [0,T]$,
	\begin{align*}
		\frac{\dd}{\dd t} \langle u_t, \varphi \rangle 
		&=  \langle u_0, \Delta e^{t\Delta}\varphi \rangle + \chi \langle (K\ast u_t)u_t, \nabla \varphi \rangle + \chi \int_0^t \langle  (K\ast u_s)u_s, e^{(t-s)\Delta}\nabla (\Delta \varphi) \rangle \, \dd s \\ 
		&\quad+ \langle \nu u_t - \mu u_t^2,\varphi \rangle + \int_0^t \langle  \nu u_s - \mu u_s^2,   e^{(t-s)\Delta} \Delta \varphi \rangle \, \dd s \\ 
		&=  \langle e^{t\Delta} u_0, \Delta \varphi \rangle + \chi \langle (K\ast u_t)u_t, \nabla \varphi \rangle - \chi \int_0^t \langle  \nabla \cdot e^{(t-s)\Delta}((K\ast u_s)u_s), \Delta  \varphi \rangle \, \dd s \\ 
		&\quad+ \langle \nu u_t - \mu u_t^2,\varphi \rangle + \int_0^t \langle e^{(t-s)\Delta} (\nu u_s - \mu u_s^2), \Delta  \varphi \rangle \, \dd s \\ 
		&= \langle u_t, \Delta \varphi \rangle + \chi \langle (K\ast u_t)u_t, \nabla \varphi \rangle + \langle \nu u_t - \mu u_t^2,\varphi \rangle.
	\end{align*}
	This yields, for all $t \in [0,T]$,
	 \begin{equation}\label{KS_weak_solution_proof}
		\langle u_t ,\varphi \rangle = \langle u_0, \varphi\rangle +  \int_0^t \langle u_s, \Delta \varphi \rangle \, \dd s + \chi \int_0^t \langle (K\ast u_s)u_s, \nabla \varphi \rangle \, \dd s+ \int_0^t \langle \nu u_s - \mu u_s^2,\varphi \rangle \, \dd s,
	\end{equation}
	which is precisely the definition of a weak solution to \eqref{eq:KS_EDP}. We introduce a mollifier $\rho \in \mathcal{C}^\infty_{c}$, $\rho\geq 0$, and the mollifying sequence $(\rho_{\delta})_{\delta>0}$ given by $\rho_{\delta} = \delta^{-d} \rho(\delta^{-1} \cdot)$; and set $u^{(\delta)}_t \coloneqq u_t \ast  \rho_{\delta}$. It follows from \eqref{KS_weak_solution_proof} that for any $t \in [0,T]$,
\begin{equation}\label{KS_molified_mild_solution}
u^{(\delta)}_t  = u^{(\delta)}_0  + \int_0^t \Delta u^{(\delta)}_s \, \dd s - \chi \int_0^t \nabla \cdot ([u_s(K\ast u_s)]\ast \rho_{\delta}) \, \dd s + \nu \int_0^t u^{(\delta)}_s \, \dd s - \mu \int_0^t u_s^2 \ast  \rho_{\delta} \, \dd s.
\end{equation}
Applying the chain rule to the $\CC^1$ function $x \in \R \mapsto (x^-)^2$, we get that for any $t \in [0,T]$,
	\begin{equation*}
		([u^{(\delta)}_t]^-)^2  =  2\int_0^t [u^{(\delta)}_s]^- \left(\Delta u^{(\delta)}_s  - \chi \nabla \cdot ([u_s(K\ast u_s)]\ast \rho_{\delta}) + \nu  u^{(\delta)}_s - \mu  u_s^2 \ast  \rho_{\delta} \right)\, \dd s.
	\end{equation*}
Integrating over $\R^d$ and using the integration-by-parts formula yields 
\begin{align}\label{KS_positivity_eq1}
	\big\Vert[u^{(\delta)}_t]^-\big\Vert_{L^2}^2  \notag
	&= - 2 \int_0^t \int_{\R^d} |\nabla [u^{(\delta)}_s]^-|^2 \, \dd x \,\dd s \\\notag &\quad +  2\int_0^t \int_{\R^d} [u^{(\delta)}_s]^- \left(  - \chi \nabla \cdot ([u_s(K\ast u_s)]\ast \rho_{\delta}) + \nu  u^{(\delta)}_s - \mu  u_s^2 \ast  \rho_{\delta} \right)\, \dd x \, \dd s. \\\notag 
	&\leq  - 2 \int_0^t \int_{\R^d} |\nabla [u^{(\delta)}_s]^-|^2 \, \dd x \,\dd s +  2\int_0^t \int_{\R^d} [u^{(\delta)}_s]^- \left(  - \chi \nabla \cdot ([u_s(K\ast u_s)]\ast \rho_{\delta}) + \nu  u^{(\delta)}_s  \right)\, \dd x \, \dd s \\\notag 
	&= - 2 \int_0^t \int_{\R^d} |\nabla [u^{(\delta)}_s]^-|^2 \, \dd x \,\dd s +	2\chi \int_0^t \int_{\R^d} \nabla [u^{(\delta)}_s]^- \cdot ([u_s(K\ast u_s)]\ast \rho_{\delta}) \,\dd x \, \dd s \\ 	&\quad+ 2\nu \int_0^t \Vert [u^{(\delta)}_s]^- \Vert_{L^2}^2\, \dd s.
 \end{align}
The second term in the right-hand side of the previous equality reads
\begin{align*}
		2\chi &\int_0^t \int_{\R^d} \nabla [u^{(\delta)}_s]^- \cdot ([u_s(K\ast u_s)]\ast \rho_{\delta}) \,\dd x \, \dd s \\ 
		&=  2\chi \int_0^t \int_{\R^d} \nabla [u^{(\delta)}_s]^- \cdot (u^{(\delta)}_s(K\ast u^{(\delta)}_s))\,\dd x \, \dd s \\ 
		&\quad + 	2\chi \int_0^t \int_{\R^d} \nabla [u^{(\delta)}_s]^- \cdot ([u_s(K\ast u_s)]\ast \rho_{\delta} - u^{(\delta)}_s(K\ast u^{(\delta)}_s) ) \,\dd x \, \dd s \\ 
		&=  \chi \int_0^t \int_{\R^d} \nabla |[u^{(\delta)}_s]^-|^2 \cdot (K\ast u^{(\delta)}_s)\,\dd x \, \dd s \\ 
		&\quad + 	2\chi \int_0^t \int_{\R^d} \nabla [u^{(\delta)}_s]^- \cdot ([u_s(K\ast u_s)]\ast \rho_{\delta} - u^{(\delta)}_s(K\ast u^{(\delta)}_s) ) \,\dd x \, \dd s \\ 
		&\eqqcolon I_1 + I_2.
\end{align*}
By integration-by-parts, \eqref{eq:divK} and the fact that $u \in L^\infty([0,T];L^{\infty})$, we have  
\begin{align}\label{KS_positivity_eq2} 
I_1 \notag&= \chi \int_0^t \int_{\R^d}  |[u^{(\delta)}_s]^-|^2 u^{(\delta)}_s\,\dd x \, \dd s\\ 
&\leq \chi \sup_{s \in [0,T]} \Vert u_s \Vert_{L^{\infty}}  \int_0^t \big\Vert[u^{(\delta)}_s]^-\big\Vert_{L^2}^2 \, \dd s,
\end{align}
using the simple convolution inequality $ \lVert u^{(\delta)}_s \rVert_{L^\infty} \leq  \lVert \rho^{(\delta)} \rVert_{L^1}  \lVert u_s \rVert_{L^\infty} \leq \sup_{s\leq T} \lVert u_s \rVert_{L^\infty}$.

Concerning $I_2$, Cauchy-Schwarz's inequality ensures that \begin{align*}
	I_2 \leq  2\chi \left(\int_0^t \big\Vert \nabla [u^{(\delta)}_s]^- \big\Vert_{L^2}^2 \, \dd s\right)^{\frac12} \left(\int_0^t  \big\Vert [u_s(K\ast u_s)]\ast \rho_{\delta} - u^{(\delta)}_s(K\ast u^{(\delta)}_s) \big\Vert_{L^2}^2 \,\dd s \right)^{\frac12}.
\end{align*}
Since $T< T^*$ and $u \in \CC([0,T];L^1) \cap L^\infty([0,T];L^{\infty})$, we deduce from \eqref{KS_properties_kernel_L_infty} that $s \in [0,T] \mapsto K\ast u_s \in L^{\infty}$ is bounded and that for any $s \in [0,T]$, 
\begin{align*}
&\big\Vert [u_s(K\ast u_s)]\ast \rho_{\delta} - u^{(\delta)}_s(K\ast u^{(\delta)}_s) \big\Vert_{L^2}\\
&\quad \leq \big\Vert [u_s(K\ast u_s)]\ast \rho_{\delta} - u_s(K\ast u_s) \big\Vert_{L^2} +  \big\Vert u_s(K\ast u_s) - u^{(\delta)}_s (K\ast u^{(\delta)}_s) \big\Vert_{L^2} \underset{\delta\to 0}{\longrightarrow} 0.
\end{align*}
It follows that for any $\eps >0$, there exists $\Delta_{\eps} >0$ such that for any $\delta < \Delta_{\eps}$ and $t \in [0,T]$,
\begin{equation}\label{KS_positivity_eq3}
	I_2 \leq  \eps \left(\int_0^t \Vert \nabla [u^{(\delta)}_s]^- \Vert_{L^2}^2 \, \dd s\right)^{\frac12}.
\end{equation}
Gathering \eqref{KS_positivity_eq1}, \eqref{KS_positivity_eq2} and \eqref{KS_positivity_eq3}, we obtain that for any $\eps >0$, $\delta < \Delta_{\eps}$ and $t \in [0,T]$,
\begin{align*}
	\big\Vert[u^{(\delta)}_t]^-\big\Vert_{L^2}^2 \notag 
	&\leq - 2 \int_0^t \big\Vert \nabla [u^{(\delta)}_s]^-\big\Vert_{L^2}^2\,\dd s +\eps  \left(\int_0^t \big\Vert \nabla [u^{(\delta)}_s]^-\big\Vert_{L^2}^2\,\dd s \right)^{\frac12} \\ 
	&\quad+ (2\nu + \chi \sup_{s \in [0,T]} \Vert u_s \Vert_{L^{\infty}}) \int_0^t \big\Vert [u^{(\delta)}_s]^- \big\Vert_{L^2}^2\, \dd s.
\end{align*}
Using that the map $x \in \R_+ \mapsto \eps \sqrt{x} - 2 x$ is bounded by $\frac{\eps^2}{8}$, we have that for any $\eps >0$, $\delta < \Delta_{\eps}$ and $t \in [0,T]$,
\begin{align*}
	\big\Vert[u^{(\delta)}_t]^-\big\Vert_{L^2}^2 \notag 
	&\leq \frac{\eps^2}{8}+ (2\nu + \chi \sup_{s \in [0,T]} \Vert u_s \Vert_{L^{\infty}}) \int_0^t \big\Vert [u^{(\delta)}_s]^- \big\Vert_{L^2}^2
	\, \dd s.
\end{align*} 
Grönwall's inequality now gives the existence of a positive constant $C$ such that for any $\eps >0$, $\delta < \Delta_{\eps}$ and $t \in [0,T]$,
\begin{align*}
	\big\Vert[u^{(\delta)}_t]^-\big\Vert_{L^2}^2 \notag &\leq C\, \eps^2.
\end{align*}
By Fatou's lemma and taking the limit $\eps \to 0$, we deduce that $\big\Vert[u_t]^-\big\Vert_{L^2} =0$ for any $t \in [0,T]$, which yields the desired result.

	\bigskip

	\textbf{Step 2: Proof of \eqref{KS_a_priori_bound_L_gamma} for $p =1$.}  
	Let $(\varphi_{n})_{n\in \N}$ be a non-decreasing sequence in $\CC_{c}^\infty(\R^d)$ that converges pointwise to $1$, which satisfies the following properties:  for each $n\in \N$, $\varphi_{n}(x) = 1$ for any $x\in [-n,n]$; $\sup_{n} \lVert \nabla \varphi_{n}\rVert_{L^\infty}<\infty$; and $\sup_{n} \lVert \Delta \varphi_{n}\rVert_{L^\infty}<\infty$. For $u \in \CC([0,T];L^1) \cap L^\infty([0,T]; L^{\infty})$, by dominated convergence theorem,
	\begin{align*}
	\int_{0}^t \langle u_{s}, \Delta \varphi_{n} \rangle \, \dd s \underset{n\to\infty}\longrightarrow 0, \quad 
	\int_{0}^t \langle (K\ast u_{s}) u_{s}, \nabla \varphi_{n} \rangle \, \dd s \underset{n\to\infty}\longrightarrow 0 \\
	\text{ and } \int_{0}^t \langle \nu u_{s} - \mu u_{s}^2, \varphi_{n} \rangle \, \dd s \underset{n\to\infty}\longrightarrow \int_{0}^t \langle \nu u_{s} - \mu u_{s}^2, 1 \rangle \, \dd s.
	\end{align*}
	Hence testing $u$ against $\varphi_{n}$ in \eqref{KS_weak_solution_proof}, using the non-negativity of $u_t$ for any $t \in [0,T^*)$, and letting $n \to +\infty$ as in the previous equation, one obtains 
	\begin{equation*}
	\Vert u_t \Vert_{L^1} \leq  \Vert u_0 \Vert_{L^1} + \nu \int_0^t\Vert u_s \Vert_{L^1} \, \dd s.
	\end{equation*}
	Thus Grönwall's inequality yields \eqref{KS_a_priori_bound_L_gamma} for $p =1$ with $F(1,T) = \lVert u_{0}\rVert_{L^1} e^{\nu T}$.
	
	\bigskip

\textbf{Step 3: Proof of \eqref{KS_a_priori_bound_L_gamma} for $p \in \big(1,\frac{\chi}{(\chi - \mu)^+}\big)$.} Let us fix $T< T^*$ and $p \in (1,+\infty)$, the restriction on $p$ will be assumed later. For any $\delta>0$, Step $1$ implies that $u^{(\delta)}_t\geq 0$. Hence 
by \eqref{KS_molified_mild_solution} and the chain rule, we obtain that for any $t\in[0,T]$, 
\begin{align*}
	(u^{(\delta)}_t)^p &= (u^{(\delta)}_0)^p + p \int_0^t (u^{(\delta)}_s)^{p -1} \Delta u^{(\delta)}_s \, \dd s - p \chi \int_0^t (u^{(\delta)}_s)^{p -1} \nabla \cdot ([u_s(K\ast u_s)]\ast \rho_{\delta}) \, \dd s \\ 
	& \quad + p \int_0^t (u^{(\delta)}_{s})^{p -1} (\nu u^{(\delta)}_s - \mu u_s^2 \ast  \rho_{\delta}) \, \dd s.
\end{align*}
We already know that $u$ and $u^{(\delta)}$ are in $L^\infty([0,T];L^1\cap L^\infty)$, so except for the term with the Laplacian, all terms in the previous equality are readily integrable on $\R^d$. It thus follows that $(u^{(\delta)}_s)^{p -1} \Delta u^{(\delta)}_s$ is integrable on $[0,T]\times \R^d$, for any $p\in (1,+\infty)$. 
Hence from Fubini's Theorem, we get that
\begin{align}\label{KS_proof_estimates_PDE_1}
	\Vert u^{(\delta)}_t \Vert^p_{L^p} 
	&= \notag\Vert u^{(\delta)}_0 \Vert^p_{L^p} + p \int_0^t \int_{\R^d} (u^{(\delta)}_s)^{p -1} \Delta u^{(\delta)}_s \, \dd x \, \dd s \\ \notag 
	&\quad  -p \chi \int_0^t \int_{\R^d} (u^{(\delta)}_s)^{p -1} \nabla \cdot ([u_s(K\ast u_s)]\ast \rho_{\delta}) \, \dd x \, \dd s  \\ \notag 
	& \quad + p \int_0^t \int_{\R^d}(u^{(\delta)}_{s})^{p -1} (\nu u^{(\delta)}_s - \mu u_s^2 \ast  \rho_{\delta})\, \dd x \, \dd s \\ 
	&\eqqcolon \Vert u^{(\delta)}_0 \Vert^p_{L^p} + I_1 + I_2 + I_3.
\end{align}

Provided integration-by-parts can be applied on $I_{1}$, it reads
 \begin{equation}\label{KS_proof_estimates_PDE_2}
 I_1 = - p (p -1) \int_0^t \int_{\R^d} (u^{(\delta)}_s)^{p -2} \vert \nabla u^{(\delta)}_s \vert^2 \, \dd x \,\dd s.
 \end{equation}
Let us justify this integration-by-parts when $p \in (1,2)$, which is the most tricky case. Consider a family $(\zeta_\eps)_{\eps>0} \in \CC^{\infty}_c(\R^d)$ of functions with values in $[0,1]$ such that $\zeta_{\eps}=1$ on $B\left(0,\frac{1}{\eps}\right)$ and such that for any $j \in \{1,2\}$, 
$$\sup_{\eps>0} \Vert \nabla^j \zeta_\eps \Vert_{L^{\infty}} <+\infty.$$
Consider 
$$A^{\eps}_t \coloneqq \int_0^t \int_{\R^d} (u^{(\delta)}_s +\eps)^{p -1} \Delta (u^{(\delta)}_s\zeta_\eps) \, \dd x \, \dd s,$$
which is well defined since 
the integrand is bounded and has bounded support thanks to $\zeta_{\eps}$. 
Moreover, $A^\eps_t$ converges towards $I_1$ when $\eps \to 0$ by the dominated convergence theorem. Indeed, for $\eps\in (0,1)$, one has the following uniform-in-$\eps$ bound 
\begin{equation*}
\big| (u^{(\delta)}_s +\eps)^{p -1} \Delta (u^{(\delta)}_s\zeta_\eps)\big| \leq \big((u^{(\delta)}_s)^{p-1}+1\big) \big(|\Delta u^{(\delta)}_s| + u^{(\delta)}_s \sup_{\eps>0} \Vert \nabla^2 \zeta_\eps \Vert_{L^{\infty}}  \big),
\end{equation*}
where the right-hand side is integrable; in particular, we use that for any $q\geq 1$, $\Delta u^{(\delta)} \in L^\infty([0,T];L^q)$, which holds since $ \lVert\Delta u^{(\delta)}_{s} \rVert_{L^q} \leq \lVert \Delta \rho_{\delta} \rVert_{L^1} \lVert u^{(\delta)}_{s} \rVert_{L^q} <\infty$.  \\
Now as the support of $u^{(\delta)}_s \zeta_{\eps}$ is included in the support of $\zeta_\eps$, for any $s \in [0,T]$, the standard integration-by-parts formula gives 
\begin{equation} \label{KS_proof_estimates_PDE_justif_IPP}
A^\eps_t = - (p -1) \int_0^t \int_{\R^d} (u^{(\delta)}_s + \eps)^{p -2} \nabla u^{(\delta)}_s \cdot \left[\zeta_\eps \nabla u^{(\delta)}_s + u^{(\delta)}_s \nabla \zeta_\eps\right] \dd x \,\dd s.
\end{equation}
Using that $(u^{(\delta)}_s)^{p -1} \nabla u^{(\delta)}_s$ is integrable over $\R^d$ (recall that $\nabla \rho_{\delta} \in L^1$, so by convolution inequality, $\nabla u^{(\delta)} \in L^\infty([0,T];L^q)$ for any $q\geq 1$), the dominated convergence theorem yields 
$$\int_0^t \int_{\R^d} (u^{(\delta)}_s + \eps)^{p -2}\, u^{(\delta)}_s\, \nabla u^{(\delta)}_s \cdot \nabla \zeta_\eps \, \dd x \,\dd s \underset{\eps \to 0}{\longrightarrow}0.$$
Thus, thanks to Fatou's lemma applied on \eqref{KS_proof_estimates_PDE_justif_IPP}, we obtain that
\begin{equation*}
(p -1) \int_0^t \int_{\R^d} (u^{(\delta)}_s)^{p -2} \vert \nabla u^{(\delta)}_s \vert^2 \, \dd x\,\dd s \leq \liminf_{\eps \to 0} -A^\eps_t = -\int_0^t \int_{\R^d} (u^{(\delta)}_s)^{p -1} \Delta u^{(\delta)}_s \, \dd x\, \dd s< + \infty.
\end{equation*}
The dominated convergence theorem can thus be applied to \eqref{KS_proof_estimates_PDE_justif_IPP} when $\eps \to 0$ and ensures that \eqref{KS_proof_estimates_PDE_2} holds true. 

Let us now focus on $I_2$. Using again the integration-by-parts formula, which can be justified as in the previous paragraph, we have
\begin{align*}
	I_2 
	&=  p \chi \int_0^t \int_{\R^d} \nabla \big( (u^{(\delta)}_s)^{p -1}\big) \cdot ([u_s(K\ast u_s)]\ast \rho_{\delta} ) \, \dd x \,\dd s \\ 
	&= p \chi \int_0^t \int_{\R^d} \nabla \big( (u^{(\delta)}_s)^{p -1}\big) \cdot  (u^{(\delta)}_s(K\ast u_s)) \, \dd x \,\dd s + R^\delta_t,
\end{align*}
where 
\begin{equation}\label{KS_proof_estimates_PDE_3_bis}
	R^\delta_t \coloneqq  p \chi \int_0^t \int_{\R^d} \nabla \big((u^{(\delta)}_s)^{p -1}\big) \cdot ([u_s(K\ast u_s)]\ast \rho_{\delta}  - u^{(\delta)}_s(K\ast u_s) ) \, \dd x \,\dd s. 
\end{equation}
Then, we have 
\begin{align*}
	I_2 &= p (p -1) \chi \int_0^t \int_{\R^d} (u^{(\delta)}_s)^{p -2} \nabla u^{(\delta)}_s \cdot (u^{(\delta)}_s(K\ast u_s)) \, \dd x \,\dd s + R^\delta_t \\ 
	&=  (p -1) \chi \int_0^t \int_{\R^d}  \nabla \big( (u^{(\delta)}_s)^{p}\big) \cdot (K\ast u_s) \, \dd x \,\dd s + R^\delta_t.
\end{align*}
Recall that $\nabla\cdot K = -\delta_{0}$, see  \eqref{eq:divK}, thus by integration-by-parts we deduce that 
\begin{align}\label{KS_proof_estimates_PDE_3}
	I_2  
	&=(p -1) \chi \int_0^t \int_{\R^d}  (u^{(\delta)}_s)^{p}  u_s \, \dd x \,\dd s + R^\delta_t.
\end{align}

Gathering \eqref{KS_proof_estimates_PDE_1}, \eqref{KS_proof_estimates_PDE_2} and \eqref{KS_proof_estimates_PDE_3}, we get that 
\begin{align}\label{KS_proof_estimates_PDE_4}
	\Vert u^{(\delta)}_t \Vert^p_{L^p} 
	&= \notag\Vert u^{(\delta)}_0 \Vert^p_{L^p} + p \nu \int_0^t \Vert u^{(\delta)}_s\Vert_{L^{p}}^p\, \dd s \\ 
	& \quad+ (p -1) \chi \int_0^t \int_{\R^d}  (u^{(\delta)}_s)^{p}  u_s \, \dd x \,\dd s - p \mu \int_0^t \int_{\R^d}(u^{(\delta)}_s)^{p -1} (u_s^2 \ast  \rho_{\delta})\, \dd x \, \dd s  \\\notag  
	&\quad - p (p -1) \int_0^t \int_{\R^d} (u^{(\delta)}_s)^{p -2} \vert \nabla u^{(\delta)}_s \vert^2 \, \dd x \,\dd s + R^\delta_t.
\end{align}
Applying Cauchy-Schwarz's inequality to the expression \eqref{KS_proof_estimates_PDE_3_bis} of $R^\delta_t$, we get 
\begin{multline*}
	|R^\delta_t| \leq p (p -1) \chi \left(\int_0^t \int_{\R^d}  (u^{(\delta)}_s)^{p -2} \vert \nabla u^{(\delta)}_s \vert^2 \, \dd x \,\dd s\right)^{\frac{1}{2}}  \\ 
	\times\left(\int_0^t \int_{\R^d} (u^{(\delta)}_s)^{p -2} \big|[u_s(K\ast u_s)]\ast \rho_{\delta}  - u^{(\delta)}_s(K\ast u_s) \big|^2 \, \dd x \,\dd s \right)^{\frac{1}{2}}.
\end{multline*}
%
Assume that we can prove that $(u^{(\delta)})^{p/2 -1}  [u(K\ast u)]\ast \rho_{\delta}$ converges to $u^{p/2} (K\ast u)$ in $L^2([0,T]\times \R^d)$, and similarly that $(u^{(\delta)})^{p/2} (K\ast u)$ converges to $u^{p/2} (K\ast u)$ in $L^2([0,T]\times \R^d)$. We will then deduce by the triangle inequality that for any $\eps>0$, there exists $\Delta_{\eps}$ such that for any $\delta < \Delta_{\eps}$ and $t \in [0,T]$,
%
\begin{equation}\label{eq:epsRdelta}
	\left(\int_0^t \int_{\R^d}(u^{(\delta)}_s)^{p -2}  \big|[u_s(K\ast u_s)]\ast \rho_{\delta}  - u^{(\delta)}_s(K\ast u_s) \big|^2 \, \dd x \,\dd s \right)^{\frac{1}{2}} \leq \eps.
\end{equation}
To prove the aforementioned $L^2([0,T]\times \R^d)$ convergences, let us apply Vitali's convergence theorem. We only make the argument precise for the convergence of $(u^{(\delta)})^{p/2 -1}  [u(K\ast u)]\ast \rho_{\delta}$, the one for $(u^{(\delta)})^{p/2} (K\ast u)$ follows similarly. 
First, we have that $\big((u^{(\delta)})^{p -2}  ([u(K\ast u)]\ast \rho_{\delta})^2\big)_{\delta>0}$ is uniformly integrable and uniformly absolutely continuous, which is derived from the fact that for any Borel set of $[0,T]\times \R^d$ of the form $A\times B$,
\begin{align*}
\int_{A\times B}  (u^{(\delta)}_{s})^{p -2} \big|[u_{s}(K\ast u_{s})]\ast \rho_{\delta} \big|^2\, \dd x\, \dd s 
&\lesssim \int_{A\times B} (u^{(\delta)}_{s})^{p}\, \dd x\, \dd s\\
& \leq \int_{A} \int_{B} (u_{s})^{p}\, \dd x\, \dd s;
\end{align*}
where we used the boundedness of $K\ast u$ in the first line, which comes from \eqref{KS_properties_kernel_L_infty} and the fact that  $u$ is in $\CC([0,T]; L^1) \cap L^\infty([0,T]; L^{\infty})$; and a convolution inequality with $ \lVert\rho_{\delta}\rVert_{L^1}=1$ in the last passage.
Second, one can check similarly that $(u^{(\delta)})^{p/2 -1}  [u(K\ast u)]\ast \rho_{\delta}$ converges to $u^{p/2} (K\ast u)$ in measure. 
Thus we have obtained that for any $\delta < \Delta_{\eps}$ and $t \in [0,T]$,
$$|R^\delta_t| \leq \eps p (p -1) \chi \left( 1+ \int_0^t \int_{\R^d}  (u^{(\delta)}_s)^{p -2} \vert \nabla u^{(\delta)}_s \vert^2 \, \dd x \,\dd s\right).$$

Plugging the previous bound in \eqref{KS_proof_estimates_PDE_4}, we get that for any $\delta < \Delta_{\eps}$ and $t \in [0,T]$,
\begin{equation}\label{KS_proof_estimates_PDE_4'}
\begin{split}
	\Vert u^{(\delta)}_t \Vert^p_{L^p} 
	&\leq \Vert u^{(\delta)}_0 \Vert^p_{L^p} + p \nu \int_0^t \Vert u^{(\delta)}_s\Vert_{L^{p}}^p\, \dd s \\  
	& \quad+ (p -1) \chi \int_0^t \int_{\R^d}  (u^{(\delta)}_s)^{p}  u_s \, \dd x \,\dd s - p \mu \int_0^t \int_{\R^d}(u^{(\delta)}_s)^{p -1}  (u_s^2 \ast  \rho_{\delta})\, \dd x \, \dd s  \\  
	&\quad + (\eps p (p -1) \chi- p (p -1)) \int_0^t \int_{\R^d} (u^{(\delta)}_s)^{p -2} \vert \nabla u^{(\delta)}_s \vert^2 \, \dd x \,\dd s + \eps p (p -1) \chi. 
\end{split}
\end{equation}
Note that the previous inequality is true for any $p \in (1,+\infty)$. Taking $\eps$ small enough so that $\eps p (p -1) \chi- p (p -1) \leq 0$ and $\eps p (p -1) \chi \leq 1$, we deduce that  
\begin{align*}
	\Vert u^{(\delta)}_t \Vert^p_{L^p} 
	&\leq\Vert u^{(\delta)}_0 \Vert^p_{L^p} + p \nu \int_0^t \Vert u^{(\delta)}_s\Vert_{L^{p}}^p\, \dd s+ (p -1) \chi \int_0^t \int_{\R^d}  (u^{(\delta)}_s)^{p}  u_s \, \dd x \,\dd s 
	\\
	&\qquad- p \mu \int_0^t \int_{\R^d}(u^{(\delta)}_s)^{p -1}  (u_s^2 \ast  \rho_{\delta})\, \dd x \, \dd s  +1. 
\end{align*}
Letting $\delta \to 0$ yields, for all $t \in[0,T]$,
\begin{align*}
	\Vert u_t \Vert^p_{L^p} \leq  1+\Vert u_0 \Vert^p_{L^p} + p \nu \int_0^t \Vert u_s\Vert_{L^{p}}^p\, \dd s - p \left( \mu - \chi \frac{p -1}{p}\right) \int_0^t  \Vert u_s \Vert_{L^{p +1}}^{p +1} \,\dd s.
\end{align*}
 Now assuming that $p < \frac{\chi}{(\chi - \mu)^+}$ gives
 \begin{align*}
	\Vert u_t \Vert^p_{L^p} \leq  1+\Vert u_0 \Vert^p_{L^p} + p \nu \int_0^t \Vert u_s\Vert_{L^{p}}^p\, \dd s.
\end{align*}
Grönwall's Lemma allows to conclude  that \eqref{KS_a_priori_bound_L_gamma} holds for $p < \frac{\chi}{(\chi - \mu)^+}$, with $F(p,T) = (1+\Vert u_0 \Vert^p_{L^p}) e^{p\nu T}$.

\bigskip

\textbf{Step 4: Proof of \eqref{KS_a_priori_bound_L_gamma} for $p \in (1,+\infty)$.} 
We aim now to bootstrap the integrability obtained in Step $3$ to any value of $p$. 
Since $\mu > \frac{d-2}{d} \chi$, it comes that $\frac{\chi}{(\chi - \mu)^+}> \frac{d}{2}$. Thus by Step $3$, there exists $p_0 \in (\frac{d}{2},\frac{\chi}{(\chi - \mu)^+})$ and a locally-bounded function $T \in \R_+ \mapsto F(p_0,T) \in \R_+^*$ such that for any $T < T^*$, 
\begin{equation}\label{KS_proof_estimates_PDE_ext_1}
\sup_{t\in [0,T]} \Vert u_t \Vert_{L^{p_0}} \leq F(p_0,T).
\end{equation}
Let us fix $p > p_0$ and $T < T^*$. Let us set, for $\delta >0$ and $t \in [0,T]$, $w^{(\delta)}_t \coloneqq (u^{(\delta)}_t)^{\frac{p}{2}}.$ 
By \eqref{KS_proof_estimates_PDE_ext_1}, we have 
\begin{equation}\label{KS_proof_estimates_PDE_ext_2}
\sup_{\delta >0} \sup_{t\in [0,T]} \Vert w^{(\delta)}_t \Vert_{L^{\frac{2p_0}{p}}} \leq F(p_0,T)^{\frac{p}{2}}.
\end{equation}
Here again, we follow the approach of \cite{TelloWinkler}, except that we will avoid the use of a Poincaré inequality, which would obviously do not hold in our whole space setting. Denoting by $W^{1,2}(\R^d)$ the usual Sobolev space of square integrable functions having square integrable weak derivative, the Gagliardo-Nirenberg inequality~\cite{BrezisMironescu} yields 
\begin{align*}
	\int_{\R^d} (u^{(\delta)}_t)^{p +1} \, \dd x 
	& = \Vert w^{(\delta)}_t\Vert_{L^{\frac{2(p +1)}{p}}}^{\frac{2(p +1)}{p}} \\ \notag
	&\leq C_{GN} \Vert w^{(\delta)}_t\Vert_{W^{1,2}}^{\frac{2(p +1)}{p}a} \Vert w^{(\delta)}_t\Vert_{L^{\frac{2p_0}{p}}}^{\frac{2(p +1)}{p}(1-a)},
\end{align*}
for the Gagliardo-Nirenberg constant $C_{GN}$ that depends only on $d,p$ and $p_{0}$, and for
\begin{equation*}
a= \frac{\frac{dp}{2p_0} - \frac{dp}{2(p +1)}}{1 - \frac{d}{2} + \frac{dp}{2p_0}}.
\end{equation*}
Because $p > p _0 > \frac{d}{2} > \frac{d-2}{d}$, one has $a\in (0,1)$. Then 
\begin{align}\label{KS_proof_estimates_PDE_ext_3}
	\int_{\R^d} (u^{(\delta)}_t)^{p +1} \, \dd x 
	&\leq C_{GN}\, F(p_0,T)^{(p+1)(1-a)} \Big( \Vert \nabla w^{(\delta)}_t\Vert_{L^2}^{\frac{2(p +1)}{p}a} +  \Vert w^{(\delta)}_t\Vert_{L^2}^{\frac{2(p +1)}{p}a}\Big) \nonumber \\ 
	&\leq   C_{d,p_0,p,T} \left[ \left(\int_{\R^d} (u^{(\delta)}_t)^{p-2} \vert \nabla u^{(\delta)}_t \vert^2 \, \dd x  \right)^{\frac{p+1}{p}a} + \left(\int_{\R^d} (u^{(\delta)}_t)^{p} \, \dd x  \right)^{\frac{p+1}{p}a}\right],
\end{align}
where we now denote by $C_{d,p_0,p,T}$ a constant that depends only on $d,p,p_{0}$ and $F(p_{0},T)$ and that may change from line to line. 
Since $p_0 > \frac{d}{2}$, observe that $\frac{p +1}{p}a <1$; hence for any $\eps>0$, there exists $C_{\eps}$ such that for any $X\geq 0$, $X^{\frac{p +1}{p}a}\leq C_{\eps} + \eps X$. From this inequality and \eqref{KS_proof_estimates_PDE_ext_3}, it follows that for any $\eps >0$, there exists $C_{\eps,d,p_0,p,T}$ such that
\begin{equation}\label{eq:boundLp+1}
	\int_{\R^d} (u^{(\delta)}_t)^{p +1} \, \dd x  \leq    \eps \int_{\R^d} (u^{(\delta)}_t)^{p-2} \vert \nabla u^{(\delta)}_t \vert^2 \, \dd x   + C_{\eps,d,p_0,p,T}+  C_{d,p_0,p,T}\int_{\R^d} (u^{(\delta)}_t)^{p} \, \dd x.
\end{equation}

Now for $\eps>0$ and $\Delta_{\eps}$ such that \eqref{eq:epsRdelta} holds true whenever $\delta<\Delta_{\eps}$, recall that Equation~\eqref{KS_proof_estimates_PDE_4'} holds for any $p\in (1,+\infty)$ and therefore, combining \eqref{KS_proof_estimates_PDE_4'} with \eqref{eq:boundLp+1}, we get that for any $\delta<\Delta_{\eps}$ and $t \in [0,T]$,
\begin{align*}
	\Vert u^{(\delta)}_t \Vert^p_{L^p} 
	&\leq  \Vert u^{(\delta)}_0 \Vert^p_{L^p} + \Big(p \nu + C_{d,p_0,p,T} \frac{p (p -1)(1 - \eps  \chi)}{\eps} \Big) \int_0^t \Vert u^{(\delta)}_s\Vert_{L^{p}}^p\, \dd s  \\
	&\quad + (p -1) \chi \int_0^t \int_{\R^d}  (u^{(\delta)}_s)^{p}  u_s \, \dd x \,\dd s - p \mu \int_0^t \int_{\R^d}(u^{(\delta)}_s)^{p -1}  (u_s^2 \ast  \rho_{\delta})\, \dd x \, \dd s  \\  
	&\quad -\frac{p (p -1)(1 - \eps  \chi)}{\eps} \int_0^t \Vert u^{(\delta)}_s\Vert_{L^{p+1}}^{p+1} \, \dd s+ C_{\chi,\eps,d,p_0,p,T},
\end{align*}
where $T \in \R_+ \mapsto C_{\chi,\eps,d,p_0,p,T}$ is a locally bounded function. Letting $\delta\to 0$ in the previous inequality, we obtain  that for any $t \in [0,T]$,
\begin{align*}
	\Vert u_t \Vert^p_{L^p} &\leq  \Vert u_0 \Vert^p_{L^p} + \Big(p \nu + C_{d,p_0,p,T} \frac{p (p -1)(1 - \eps  \chi)}{\eps} \Big) \int_0^t \Vert u_s\Vert_{L^{p}}^p\, \dd s   \\
	&\quad  -\left(p\mu + \frac{p (p -1)(1 - \eps \chi)}{\eps} - (p -1)\chi \right) \int_0^t \Vert u_s\Vert_{L^{p+1}}^{p+1} \, \dd s+ C_{\chi,\eps,d,p_0,p,T}.
\end{align*}
Choosing $\eps$ small enough yields for any $t \in [0,T]$,
\begin{align*}
	\Vert u_t \Vert^p_{L^p} \leq  \Vert u_0 \Vert^p_{L^p} + \Big(p \nu + C_{d,p_0,p,T} \frac{p (p -1)(1 - \eps  \chi)}{\eps} \Big) \int_0^t \Vert u_s\Vert_{L^{p}}^p\, \dd s  + C_{\chi,\eps,d,p_0,p,T}.
\end{align*}
Grönwall's lemma thus allows to conclude that \eqref{KS_a_priori_bound_L_gamma} holds true for any $p\in(1,+\infty)$.

\bigskip

\textbf{Step 5: Proof of \eqref{KS_a_priori_bound_L_gamma} for $p = + \infty$.}
The goal is now to extend the $L^p$-bounds on $u$ proved for $p \in [1,+\infty)$, to an $L^\infty$-bound.

Let us fix $t< T^*$. In view of the mild formulation \eqref{eq:KSmild}, use that $u_t$ is non-negative almost everywhere and that the semigroup $e^{t\Delta}$ is positive to deduce
%
%
\begin{align}\label{KS_proof_estimates_PDE_6}
	\Vert u_t \Vert_{L^{\infty}} &\leq \Vert u_0 \Vert_{L^{\infty}}  +\chi  \left\Vert \int_0^t \nabla \cdot e^{(t-s)\Delta} ((K\ast u_s)u_s)  \, \dd s \right\Vert_{L^{\infty}} + \nu \int_0^t \Vert u_s \Vert_{L^{\infty}}\, \dd s.
\end{align}
Let $p >d$ and $\alpha \in (\frac{d}{p}, 1)$. Using the classical Sobolev embedding $H^{\alpha}_p(\R^d) \hookrightarrow L^{\infty}(\R^d)$, we get
\begin{align*}
\left\Vert \int_0^t \nabla \cdot e^{(t-s)\Delta} ((K\ast u_s)u_s)  \, \dd s \right\Vert_{L^{\infty}}  
&\lesssim \left\Vert (I - \Delta)^{\frac{\alpha}{2}} \int_0^t \nabla \cdot e^{(t-s)\Delta} ((K\ast u_s)u_s)  \, \dd s \right\Vert_{L^{p}} \\ 
&\lesssim \int_0^t   \left\Vert (I - \Delta)^{\frac{\alpha}{2}} \nabla \cdot e^{(t-s)\Delta} ((K\ast u_s)u_s) \right\Vert_{L^{p}} \, \dd s  \\ 
&\lesssim \int_0^t  \left\Vert (I - \Delta)^{\frac{\alpha}{2}} \nabla  g_{2(t-s)} \right\Vert_{L^q}  \Vert (K\ast u_s)u_s \Vert_{L^{\tilde q}} \, \dd s ,
\end{align*}
where the last passage follows by a convolution inequality with $1+\frac{1}{p} = \frac{1}{q} + \frac{1}{\tilde q}$. For the following estimates, we want $q$ close to $1$, but $q>1$ is imposed for the interpolation inequalities we will use. Hence, let us choose $p>2d$ so that one can choose $\alpha < 1/2$; and fix $q<\frac{4d}{4d-1}$.
Now since $\Vert (I - \Delta)^{\frac{\alpha}{2}} \nabla  g_{2(t-s)} \Vert_{L^q} = \lVert \nabla  g_{2(t-s)} \rVert_{\alpha,q}$ and $H^{\alpha}_{q}$ coincides with the Triebel-Lizorkin space $F^\alpha_{q,2}$ when $1<q<\infty$ (see \cite[p.14]{RunstSickel}), the interpolation inequality for such spaces (\cite[p.87]{RunstSickel}) yields
\begin{align*}
\lVert \nabla  g_{2(t-s)} \rVert_{\alpha,q} 
&\leq \lVert \nabla  g_{2(t-s)} \rVert_{L^q}^{1-\alpha} \, \lVert \nabla  g_{2(t-s)} \rVert_{1,q}^\alpha \\
&\lesssim \lVert \nabla  g_{2(t-s)} \rVert_{L^q}^{1-\alpha} \, \big(\lVert \nabla  g_{2(t-s)} \rVert_{L^q} + \lVert \nabla^2  g_{2(t-s)} \rVert_{L^q}\big)^\alpha \\
& \lesssim (1\wedge (t-s))^{-(\frac{1}{2} + \frac{\alpha}{2} + d(1-\frac{1}{q}))} ,
\end{align*}
where we used, in the second inequality, the equivalence of norms between $H^1_{q}$ and the usual Sobolev space $W^{1,q}(\R^d)$ (see \cite[p.14]{RunstSickel}); and using the Gaussian estimates
 \eqref{eq:heat-estimate} in the last inequality. Since $\alpha < 1/2$ and $q<\frac{4d}{4d-1}$, we get that $\frac{1}{2} + \frac{\alpha}{2} + d(1-\frac{1}{q})\eqqcolon \eta \in (\frac{1}{2},1)$.
Thus
\begin{align}\label{KS_proof_estimates_PDE_7}
\left\Vert \int_0^t \nabla \cdot e^{(t-s)\Delta} ((K\ast u_s)u_s)  \, \dd s \right\Vert_{L^{\infty}}  
&\lesssim \int_0^t  (1\wedge (t-s))^{-\eta} \Vert (K\ast u_s)u_s \Vert_{L^{\tilde q}} \, \dd s \nonumber\\
&\lesssim C_{T,\eta}\, \sup_{s \in [0,t]} \Vert K\ast u_s \Vert_{L^{\infty}}\,  \sup_{s \in [0,t]}\Vert u_s \Vert_{L^{\tilde q}}.
\end{align}
Hence, in view of \eqref{Bound_K_convol_L_infty} and \eqref{KS_a_priori_bound_L_gamma}, we deduce that for $r\in (d,+\infty)$,
\begin{align*}
 \sup_{s \in [0,t]} \Vert K\ast u_s \Vert_{L^{\infty}} \leq C_{K,d} \sup_{s \in [0,t]} \lVert u_{s} \rVert_{L^{1}\cap L^r}
 \lesssim (1 + F(r,T)) <\infty ;
\end{align*}
and 
\begin{align*}
\sup_{s \in [0,t]}\Vert u_s \Vert_{L^{\tilde q}} \leq F(\tilde{q}, T).
\end{align*}
Plugging these two bounds in \eqref{KS_proof_estimates_PDE_7}, using that $\frac{\alpha +1}{2}< 1$ and finally coming back to \eqref{KS_proof_estimates_PDE_6}, we get that for any $t \in [0,T^*)$,
\begin{align*}
	\Vert u_t \Vert_{L^{\infty}} 
	\leq \Vert u_0 \Vert_{L^{\infty}}  +  C_{T,\eta} F(\tilde q, T)\, (1+ F(r,T)) + \nu \int_0^t \Vert u_s \Vert_{L^{\infty}}\, \dd s.
\end{align*}
Applying Grönwall's Lemma, we conclude that \eqref{KS_a_priori_bound_L_gamma} holds true for $p=+\infty$. 
\end{proof}


\section{Convergence of the moderately interacting particle system to the Keller--Segel PDE}
\label{sec:particles}

The goal of this section is to prove Theorem~\ref{th:convergence_moderate}. We proceed as follows: in Subsection~\ref{subsec:eqempmeas}, we establish the equations satisfied by the empirical measure $\mu^N$ and the mollified empirical measure $u^N$; then in Subsection~\ref{subsec:boundparticles}, we use these equations to bound the moments of the mass of the system; Subsection~\ref{subsec:stocint} gathers results on the moments of the stochastic integrals $M^N$ and $J^N$ defined in \eqref{KS_mollified_empirical_mesure_equation_mild} and contains proofs for the latter. Then in Subsection~\ref{subsec:apriorimollempmeas}, we establish an \emph{a priori} bound on the moments of $u^N$, which will be used to carry out the proof of Theorem~\ref{th:convergence_moderate} in Subsection~\ref{subsec:proofThm2}. Finally we prove Corollary~\ref{cor:KR-convergence} in Subsection~\ref{subsec:proofCor}.

\subsection{Equation satisfied by the empirical measure}
\label{subsec:eqempmeas}

The goal of this section is to derive the equations satisfied by $\mu^N$ and $u^N$, see respectively Equation~\eqref{weak_formulation_empirical_measure} and Equation~\eqref{KS_mollified_empirical_mesure_equation_mild}. Let ${\phi : \R^+ \times \R^d \to \R}$ be a function which is $\CC^1$ with respect to time and $\CC^2$ with respect to space, and set $\phi(t,\triangle) = 0$ for any $t$. Recalling the construction of the interacting particle system from Section~\ref{subsubsec:constructionIPS} and thanks to Itô's formula over random time intervals, we deduce that for any $k \in \Lambda^N$,
\begin{align*}
\phi(t,X^k_t) 
&= \phi(T^k_0,X^k_{T^k_0}) \1_{\{t \geq T^k_0\}} - \phi(T^k_1,X^k_{T^k_1}) \1_{\{t \geq T^k_1\}} + \chi \int_0^t \1_{\{s \in L^k\}} \nabla \phi (s,X^k_s) \cdot F_A(K\ast u^N_s(X^k_s)) \, \dd s \\ 
&\quad +  \int_0^t \1_{\{s \in L^k\}} \pa_s \phi (s,X^k_s) \, \dd s + \sqrt{2}\int_0^t \1_{\{s \in L^k\}}  \nabla \phi (s,X^k_s)\cdot \dd B^k_s + \int_0^t \1_{\{s \in L^k\}}  \Delta \phi (s,X^k_s) \, \dd s,
\end{align*}
where we recall that $L^k = [T^k_{0},T^k_{1})$ is the interval of existence of particle $k$.
By summing over $k \in \Lambda^N$ and since $\phi(t,\triangle)=0$, it follows that
\begin{align}\label{eq1_dynamics_particles}
\langle \mu^N_t, \phi(t,\cdot) \rangle \notag
&= \frac1N \sum_{k \in \Lambda^N}\phi(T^k_0,X^k_{T^k_0}) \1_{\{t \geq T^k_0\}} - \frac1N \sum_{k\in \Lambda^N}\phi(T^k_1,X^k_{T^k_1}) \1_{\{t \geq T^k_1\}} \\ 
&\quad + \int_0^t \langle \mu^N_s,\pa_s\phi (s,\cdot)\rangle \, \dd s  + \chi\int_0^t \langle \mu^N_s, \nabla \phi (s,\cdot) \cdot F_A(K\ast u^N_s(\cdot)) \rangle \, \dd s \\
 \notag 
&\quad  + \int_0^t \langle \mu^N_s,\Delta \phi (s,\cdot)\rangle \, \dd s  + \frac{\sqrt{2}}{N} \sum_{k \in \Lambda^N}\int_0^t \1_{\{s \in L^k\}}  \nabla \phi (s,X^k_s) \cdot  \dd B^k_s.
\end{align}
Now observe that 
\begin{align*}
\frac1N \sum_{k \in \Lambda^N}\phi(T^k_0,X^k_{T^k_0}) \1_{\{t \geq T^k_0\}}  &= \frac1N \sum_{k_0=1}^N\phi(0,\xi^{k_0}) + \frac1N \sum_{k \in \Lambda^N\backslash\{ (k_0), \, 1\leq k_0 \leq N\}}\phi(T^k_0,X^k_{T^k_0}) \1_{\{t \geq T^k_0\}} .
\end{align*}
Using the fact that if a particle $k \in \Lambda^N\backslash\{ (k_0), \, 1\leq k_0 \leq N\}$ was born before time $t$, then its mother $(k,-)$ must have divided before time $t$ and given birth to two particles at the place where it died, we find that 
\begin{align*}
&\frac1N \sum_{k \in \Lambda^N\backslash\{ (k_0), \, 1\leq k_0 \leq N\}} \phi(T^k_0,X^k_{T^k_0}) \1_{\{t \geq T^k_0\}}\\
&= \frac1N \sum_{k \in \Lambda^N\backslash\{ (k_0), \, 1\leq k_0 \leq N\}} \phi\big(T^{(k,-)}_1,X^{(k,-)}_{T^{(k,-)}_1}\big) \1_{\{t \geq T^{(k,-)}_1\}} \1_{\{T^{(k,-),\text{div}}_1 <T^{(k,-),\text{die}}_1\}}
	\\ 
&=\frac2N \sum_{k \in \Lambda^N}\phi(T^k_1,X^k_{T^k_1}) \1_{\{t \geq T^k_1\}} \1_{\{T^{k,\text{div}}_1 <T^{k,\text{die}}_1\}}.
\end{align*}
It follows that 
\begin{equation}\label{eq2_dynamics_particle}
\begin{split}
&\frac1N \sum_{k \in \Lambda^N}\phi(T^k_0,X^k_{T^k_0}) \1_{\{t \geq T^k_0\}} - \frac1N \sum_{k\in \Lambda^N}\phi(T^k_1,X^k_{T^k_1}) \1_{\{t \geq T^k_1\}} \\ 
&= \frac1N \sum_{k_0=1}^N\phi(0,\xi^{k_0}) +  \frac1N \sum_{k \in \Lambda^N}\phi(T^k_1,X^k_{T^k_1}) \1_{\{t \geq T^k_1\}} \1_{\{T^{k,\text{div}}_1 <T^{k,\text{die}}_1\}} \\
&\quad - \frac1N \sum_{k \in \Lambda^N}\phi(T^k_1,X^k_{T^k_1}) \1_{\{t \geq T^k_1\}} \1_{\{T^{k,\text{div}}_1 \geq T^{k,\text{die}}_1\}}.
\end{split}
\end{equation}
We will rewrite this identity in terms of the underlying random Poisson measure that was used to define the stopping times $T^{k,\text{div}}_1$ and $T^{k,\text{die}}_1$ in \eqref{eq:defT1}. Namely, recall that $\NN$ is random measure on $\R^+ \times \Lambda^N \times \R^+$ with intensity measure $\lambda \otimes m \otimes \lambda,$ where $m$ is the counting measure on $\Lambda^N$. Since the particle $k \in \Lambda^N$ divides with a constant rate $\nu$ and dies at time $s$ with an instantaneous rate $\mu (u^N_s(X^k_s)\wedge A)$, we have 
\begin{align*}
& \frac1N \sum_{k \in \Lambda^N}\phi(T^k_1,X^k_{T^k_1}) \1_{\{t \geq T^k_1\}} \1_{\{T^{k,\text{div}}_1 <T^{k,\text{die}}_1\}} 
- \frac1N \sum_{k \in \Lambda^N}\phi(T^k_1,X^k_{T^k_1}) \1_{\{t \geq T^k_1\}} \1_{\{T^{k,\text{div}}_1 \geq T^{k,\text{die}}_1\}}\\ 
&=  \frac1N \int_0^t \int_{\Lambda^N} \int_{\R_+} \1_{\{k \in I_s\}} \phi(s,X^k_{s^-}) \Big(\1_{\{\rho < \nu\}} - \1_{\big\{\nu\leq \rho < \nu + \mu u^{N}_s(X^k_{s^-})\wedge A\big\}}\Big)\,  \NN (\dd s,\dd k,\dd \rho).
\end{align*}
Plugging this in \eqref{eq2_dynamics_particle} and coming back to \eqref{eq1_dynamics_particles}, we deduce that 
\begin{equation}
\label{eq:muN}
\begin{split}
\langle& \mu^N_t, \phi(t,\cdot) \rangle \\
&=  \langle \mu^N_0,\phi(0,\cdot)\rangle + \int_0^t \langle \mu^N_s,\pa_s\phi (s,\cdot)\rangle \, \dd s  + \chi\int_0^t \langle \mu^N_s, \nabla \phi (s,\cdot) \cdot F_A(K\ast u^N_s(\cdot)) \rangle \, \dd s \\ 
&\quad  + \int_0^t \langle \mu^N_s,\Delta \phi (s,\cdot)\rangle \, \dd s  + \frac{\sqrt{2}}{N} \sum_{k \in \Lambda^N}\int_0^t \1_{\{s \in L^k\}}  \nabla \phi (s,X^k_s) \cdot \dd B^k_s \\
& \quad + \frac1N \int_0^t \int_{\Lambda^N} \int_{\R_+} \1_{\{k \in I_s\}} \phi(s,X^k_{s^-}) \Big(\1_{\{\rho < \nu\}} - \1_{\big\{\nu\leq \rho < \nu + \mu u^{N}_s(X^k_{s^-})\wedge A\big\}} \Big)\, \NN (\dd s,\dd k,\dd \rho).
\end{split}
\end{equation}
Let us now introduce the compensated Poisson random measure $\NNN$ defined by 
\begin{equation*}
\NNN (\dd s, \dd k, \dd \rho) \coloneqq \NN (\dd s, \dd k, \dd \rho) - \dd s\otimes m(\dd k) \otimes\dd \rho.
\end{equation*}
Since we have 
\begin{align*}
&\frac1N \int_0^t \int_{\Lambda^N} \int_{\R_+} \1_{\{k \in I_s\}} \phi(s,X^k_{s^-}) \Big(\1_{\{\rho < \nu\}} - \1_{\big\{\nu\leq \rho < \nu + \mu u^{N}_s(X^k_{s^-})\wedge A\big\}}\Big) \dd s\,  m(\dd k) \, \dd \rho \\ 
& = \int_0^t \frac1N \sum_{k \in I_s} \phi(s,X^k_s) (\nu - \mu (u^N_s(X^k_s) \wedge A)) \, \dd s  \\ &=\int_0^t \langle \mu^N_s, \phi(s,\cdot) (\nu - \mu (u^N_s(\cdot) \wedge A)) \rangle \, \dd s,
\end{align*} 
it comes that Equation~\eqref{eq:muN} rewrites using $\NNN$ as 
\begin{align}
\label{weak_formulation_empirical_measure} 
\notag\langle \mu^N_t, \phi(t,\cdot) \rangle 
&=  \langle \mu^N_0,\phi(0,\cdot)\rangle + \int_0^t \langle \mu^N_s,\pa_s\phi (s,\cdot)\rangle \, \dd s  + \chi \int_0^t \langle \mu^N_s, \nabla \phi (s,\cdot) \cdot F_A(K\ast u^N_s(\cdot)) \rangle \, \dd s \\ \notag 
&\quad  + \int_0^t \langle \mu^N_s,\Delta \phi (s,\cdot)\rangle \, \dd s  + \int_0^t \langle \mu^N_s, \phi(s,\cdot) (\nu - \mu (u^N_s(\cdot) \wedge A)) \rangle \, \dd s \\ 
&\quad + \frac{\sqrt{2}}{N} \sum_{k \in \Lambda^N}\int_0^t \1_{\{s \in L^k\}}  \nabla \phi (s,X^k_s) \cdot \dd B^k_s \\ \notag 
& \quad + \frac1N \int_0^t \int_{\Lambda^N} \int_{\R_+} \1_{\{k \in I_s\}} \phi(s,X^k_{s^-}) \Big(\1_{\{\rho < \nu\}} - \1_{\big\{\nu\leq \rho < \nu + \mu u^{N}_s(X^k_{s^-})\wedge A\big\}}\Big)\, \NNN (\dd s,\dd k,\dd \rho).
\end{align}
To derive the equation satisfied by the mollified empirical measure $u^N$, we apply the previous equality to $\phi_{t,x}(s,y) = e^{(t-s)\Delta} \theta^N(x-y)$, for fixed $t\in [0,T]$ and $x \in \R^d$. Then we get the following mild formulation
%
\begin{equation}\label{KS_mollified_empirical_mesure_equation_mild}
	\begin{split}
		u^N_t(x) &= e^{t\Delta}u^N_0(x) - \chi \int_0^t \nabla \cdot e^{(t-s)\Delta} \langle \mu^N_s, K\ast u^N_s(\cdot) \theta^N(x-\cdot) \rangle \, \dd s  \\ 
		&\quad + \int_0^t e^{(t-s)\Delta}\langle \mu^N_s, (\nu  - \mu u^N_s\wedge A)\theta^N(x-\cdot) \rangle \, \dd s \\  
		& \quad - M^N_t + J^N_t, 
	\end{split}
\end{equation}
where 
\begin{equation}\label{KS_brownian_conv_integral_def}
M^N_t \coloneqq \frac{\sqrt{2}}{N} \sum_{k \in \Lambda^N} \int_0^t  \1_{\{k \in I_s\}} e^{(t-s)\Delta}\nabla \theta^N (x-X^k_s) \cdot  \dd B^k_s,
\end{equation} 
and 
\begin{equation}\label{KS_Poisson_conv_integral_def}
	J^N_t \coloneqq \frac1N \int_0^t \int_{\Lambda^N} \int_{\R_+} \1_{\{k \in I_s\}} e^{(t-s)\Delta}\theta^N(x-X^k_{s^-}) \Big(\1_{\{\rho < \nu\}} - \1_{\big\{\nu\leq \rho < \nu + \mu u^{N}_s(X^k_{s^-})\wedge A\big\}}\Big) \, \NNN (\dd s,\dd k,\dd \rho).
\end{equation}

\subsection{Control of the number of particles in the system}
\label{subsec:boundparticles}

We aim to give a control on the number of particles alive at time $t$ which is given, for any $t \in [0,T]$, by 
\begin{equation*}
\mathfrak{N}^N_t \coloneqq N\langle  \mu^N_t, 1 \rangle = \text{Card}(I_t),
\end{equation*}
or equivalently on the renormalised number of particles alive at time $t$ with respect to the initial size $N$ of the population:
\begin{equation*}
m^N_t \coloneqq \langle  \mu^N_t, 1 \rangle.
\end{equation*}

\begin{proposition}
\label{KS_prop_moment_k_number_ptcles}
	For any $q \geq 1$, we have
	\begin{equation*}
		\sup_{N \geq 1}\sup_{t \in [0,T]} \E |m^N_t|^q = \sup_{N \geq 1}\sup_{t \in [0,T]} \E \Vert u^N_t\Vert_{L^1}^q < + \infty.
	\end{equation*}
\end{proposition}

\begin{proof}
 Taking $\phi =1$ in \eqref{weak_formulation_empirical_measure}, we get that for any $t \in [0,T]$,
\begin{align*}
	m^N_t \notag &= 1 +\frac1N \int_0^t \int_{\Lambda^N} \int_{\R_+} \1_{\{k \in I_s\}}  \Big(\1_{\{\rho < \nu\}} - \1_{\big\{\nu\leq \rho < \nu + \mu u^{N}_s(X^k_{s^-})\wedge A\big\}}\Big)\, \NN (\dd s,\dd k,\dd \rho) .
\end{align*}
Applying Itô's formula for this Poisson random integral and for the continuous function $|\cdot|^q$ with $q \geq 1$ (see Lemma $4.4.5$ in \cite{Applebaum}), we get that for any $t \in [0,T]$,
\begin{align*}
	&\E |m^N_t|^q \\ 
	&= 1 + \sum_{k \in \Lambda^N}\int_0^t \int_{\R_+} \E\Big[\Big\vert m^N_{s^-} + \frac{1}{N}\1_{\{k \in I_s\}}\Big(\1_{\{\rho < \nu\}} - \1_{\big\{\nu \leq \rho \leq \nu + \mu u^N_s(X^{k}_{s^-})\wedge A\big\}}\Big) \Big\vert^q  - \left\vert m^N_{s^-} \right\vert^q \Big] \, \dd \rho \, \dd s \\ 
	&=  1 + \sum_{k \in \Lambda^N}\int_0^t \int_0^{\nu+\mu A} \E\Big[\Big\vert m^N_{s} + \frac{1}{N}\1_{\{k \in I_s\}}\Big(\1_{\{\rho < \nu\}} - \1_{\big\{\nu \leq \rho \leq \nu + \mu u^N_s(X^{k}_{s})\wedge A\big\}}\Big) \Big\vert^q  - \left\vert m^N_{s} \right\vert^q \Big] \, \dd \rho \, \dd s .
\end{align*}
By the mean-value theorem applied to $\vert \cdot \vert^q$, we get that  
\begin{align*}
	\E |m^N_t|^q &\leq 1+C q\sum_{k \in \Lambda^N}\int_0^t  \int_0^{\nu +\mu A} \E\left[ \big(\left\vert m^N_{s}\right\vert^{q-1} + 1\big) \frac1N \1_{\{k \in I_s\}} \right] \dd \rho \, \dd s \\ 
	&\leq  1+C q\int_0^t \E\left[ \left\vert m^N_{s}\right\vert^{q} +  |m^N_s|\right] \dd s.
	\end{align*}
Grönwall's lemma yields the conclusion. 
\end{proof}

\subsection{Estimates on the stochastic convolution integrals}
\label{subsec:stocint}

In this section, we state and prove Bessel and $L^p$ bounds on the infinite-dimensional stochastic convolution integrals $M^N$ and $J^N$, defined by \eqref{KS_brownian_conv_integral_def} and \eqref{KS_Poisson_conv_integral_def} respectively. Concerning the Brownian integral $M^N$, we recall the following couple of results:
First, an estimate in $H^\gamma_{2}$ is known, see \cite[Propositions~A.8]{Pisa}.
\begin{proposition}\label{KS_estimate_Bessel_Brownian_integral}
	Let $p\geq 1$ and $m \geq 1$. For any $\eps >0$, there exists a positive constant $C= C_{T,d,m,\eps,A,p}$ such that for any $t \in [0,T]$ and $N \geq 1$,
	\begin{equation*}
	\big\Vert \Vert (I-\Delta)^{\frac{\gamma}{2}}M^N_t\Vert_{L^2} \big\Vert_{L^m_\Omega} \leq C N^{-\frac{1}{2} \big( 1 -\alpha (d + 2\gamma  )\big) + \eps}.
	\end{equation*}
\end{proposition}
By Sobolev embedding (see \eqref{eq:Bessel_embedding}), one can deduce from the previous proposition the following result, see \cite[Propositions~A.10]{Pisa}.
\begin{proposition}\label{KS_estimate_Lp_Brownian_integral}
	Let $p\geq 1$ and $m \geq 1$. For any $\eps >0$, there exists a positive constant $C= C_{T,d,m,\eps,A,p}$ such that for any $t \in [0,T]$ and $N \geq 1$,
	\begin{equation*}
	\big\Vert \Vert M^N_t\Vert_{L^p} \big\Vert_{L^m_\Omega} \leq C N^{-\frac{1}{2} \big( 1 -\alpha (d + \varkappa_p )\big) + \eps},
	\end{equation*}
	where $\varkappa_p \coloneqq \max \left(0, d \big(1 - \frac2p\big)\right).$
\end{proposition}

Let us now focus on the estimates of the Poisson stochastic integral $J^N$, which is the main difference in comparison with \cite{Pisa}. We first recall the following simple estimate, taken for example from \cite[Lemma 15]{FlandoliLeimbachOlivera}. 
\begin{lemma}
\label{lem:BesselMoll}
	For any $\gamma \geq 0$, there exists $C>0$ such that for all $N \geq 1$,
	\begin{equation*}
	\Vert (I-\Delta)^{\frac{\gamma}{2}} \theta^N \Vert_{L^2} \leq C N^{\alpha\left(\gamma + \frac{d}{2}\right)}.
	\end{equation*}
\end{lemma}
As for $M^N$, we start with the Bessel estimate.

\begin{proposition}\label{KS_estimate_Bessel_Poisson_integral}
	Let $\gamma \geq 0$ and $m = 2^n$, for some $n \in \N$. For any $\eps >0$, there exists a positive constant $C= C_{T,d,m,\eps,A,\gamma}$ such that for all $N \geq 1$,
	\begin{equation*}
	 \big\Vert \Vert (I-\Delta)^{\frac{\gamma}{2}} J^N_t \Vert_{L^2} \big\Vert_{L^m_{\Omega}} \leq C N^{-\frac12 \left( 1 -\alpha\left(d + 2\gamma - \frac{4}{m} \right)\right) + \eps}.
	\end{equation*}
\end{proposition}

\begin{remark}
Note that the power of $N$ can be chosen to be negative for $\eps$ small enough if $ \gamma < \frac{1}{2\alpha} (1 - \alpha d)$. 
\end{remark}

\begin{proof}
	By the stochastic Fubini theorem for compensated Poisson integrals (see \cite[Theorem $1.1$]{Stochastic_Fubini}), we have 
	\begin{align*}
	&(I-\Delta)^{\frac{\gamma}{2}} J^N_t\\ 
	 &=\frac1N \int_0^t \int_{\Lambda^N} \int_{\R_+} \1_{\{k \in I_s\}} (I-\Delta)^{\frac{\gamma}{2}}e^{(t-s)\Delta}\theta^N(\cdot-X^k_{s^-}) (\1_{\{\rho < \nu\}} - \1_{\{\nu\leq \rho < \nu + \mu u^{N}_s(X^k_{s^-})\wedge A\}}) \NNN (\dd s,\dd k,\dd \rho).
	\end{align*}
	 Then, the functional BDG inequality for jump processes (see \cite[Corollary $2.14$]{Hausenblas}) ensures that there exists a constant $C>0$, which may change from line to line, depending only on $m$, $d$ and the cutoff parameter $A$, such that
\begin{align*}
	&\E \Vert (I-\Delta)^{\frac{\gamma}{2}} J^N_t\Vert_{L^2}^m \\ 
	&\leq C \sum_{l=1}^n  \E \Big|  \sum_{k \in \Lambda^N}\int_0^t  \int_0^{\nu +\mu A}  \Big\Vert \frac1N \1_{\{k \in I_s\}} (I-\Delta)^{\frac{\gamma}{2}}e^{(t-s)\Delta}\theta^N(\cdot-X^k_{s^-}) \\ 
	& \hspace{5,5cm} \times (\1_{\{\rho < \nu\}} - \1_{\{\nu\leq \rho < \nu + \mu u^{N}_s(X^k_{s^-})\wedge A\}}) \Big\Vert_{L^2}^{2^l}\, \dd s  \,\dd \rho \Big|^{2^{n-l}}
	\\ 
	&\leq \frac{C}{N^m} \sum_{l=1}^n  \E \Big| \sum_{k \in \Lambda^N}\int_0^t    \Big\Vert\1_{\{k \in I_s\}} (I-\Delta)^{\frac{\gamma}{2}}e^{(t-s)\Delta}\theta^N(\cdot-X^k_{s^-}) \Big\Vert_{L^2}^{2^l}\, \dd s  \Big|^{2^{n-l}}
	\\ 
	&\leq \frac{C}{N^m} \sum_{l=1}^n  \E\big[ (\mathfrak{N}^N_t)^{2^{n-l}}\big]
	\left(  \int_0^t \left\Vert (I-\Delta)^{\frac{1-\eps}{m}}e^{(t-s)\Delta} \right\Vert_{L^2 \to L^2}^{2^{l}} \left\Vert  (I-\Delta)^{\frac{\gamma}{2} - \frac{1-\eps}{m}}\theta^N \right\Vert_{L^2}^{2^l}\, \dd s  \right)^{2^{n-l}}.
	\end{align*}
Hence upon using \eqref{eq:Besov-heat-estimate} and Lemma~\ref{lem:BesselMoll}, it comes
	\begin{align*}
	\E \Vert (I-\Delta)^{\frac{\gamma}{2}} J^N_t\Vert_{L^2}^m 
	&\leq \frac{C}{N^m} \sum_{l=1}^n  \E\big[ (\mathfrak{N}^N_t)^{2^{n-l}}\big]
	\left(  \int_0^t (t-s)^{-(\frac{1-\eps}{m})2^{l}}  N^{2^l\alpha\left(\gamma - 2\frac{1-\eps}{m} + \frac{d}{2}\right)}\, \dd s  \right)^{2^{n-l}} 
	\\ 
	&\leq \frac{C}{N^m} \E\big[ (\mathfrak{N}^N_t)^{\frac{m}{2}}\big] N^{m\alpha\left(\gamma - 2\frac{1-\eps}{m} + \frac{d}{2}\right)}
	\\ 
	&\leq C N^{-\frac{m}{2} \left( 1 -\alpha\left(d + 2\gamma - 4\frac{1-\eps}{m}\right)\right)},
\end{align*}
where the last inequality follows from Proposition \ref{KS_prop_moment_k_number_ptcles}. We conclude because the preceding inequality is true for any $\eps >0$. 
\end{proof} 
Concerning the $L^p$ estimates, we have the following result.

\begin{proposition}\label{KS_estimate_Lp_Poisson_integral}
	Let us fix $p\geq 1$ and $m = 2^n$ for $n \geq 1$. For any $\eps >0$, there exists a positive constant $C= C_{T,d,n,\eps,A,p}$ such that for any $t \in [0,T]$ and $N \geq 1$,
	\begin{equation*}
	\big\Vert \Vert J^N_t\Vert_{L^p} \big\Vert_{L^m(\Omega)} \leq C N^{-\frac{1}{2} \left( 1 -\alpha\left(d + \varkappa_p - \frac{4}{m} \right)\right) + \eps},
	\end{equation*}
	where we recall that $\varkappa_p = \max \left(0, d \big(1 - \frac2p\big)\right).$
\end{proposition}

Before proving Proposition \ref{KS_estimate_Lp_Poisson_integral}, we state and prove the following lemma on the control of the moments of the particles, which will be needed in the proof of Proposition~\ref{KS_estimate_Lp_Poisson_integral}. It somehow extends Proposition~\ref{KS_prop_moment_k_number_ptcles}, which would correspond to $\kappa=0$ below, with a supremum in time within the expectation.

\begin{lemma}\label{Lemma_control_moment_ptcles}
	Let $\kappa \geq 2$ and 
	$m\geq 2$, 
	and assume that \ $$\sup_{N \geq 1}\E \Big[\langle \mu^N_0, |\cdot|^{\kappa} \rangle^m\Big]  < +\infty.$$ 
	Then, we have 
\begin{equation*}
		\sup_{N \geq 1}\E \Big[\sup_{s \in [0,T]} \langle \mu^N_s, |\cdot|^{\kappa} \rangle^m\Big]  < +\infty.
	\end{equation*} 
\end{lemma}

\begin{proof}
	By \eqref{weak_formulation_empirical_measure} applied to $\phi = |\cdot|^\kappa$, we have that for any $t \in [0,T]$,
	\begin{align*}
		\langle \mu^N_t, |\cdot|^{\kappa} \rangle\notag 
		&= \langle \mu^N_0 , |\cdot|^{\kappa}  \rangle + \int_0^t \langle \mu^N_s, \Delta  |\cdot|^{\kappa}  \rangle \, \dd s  + \chi \int_0^t \langle \mu^N_s, F_A(K\ast u^N_s) \cdot \nabla  |\cdot|^{\kappa} \rangle \, \dd s \\\notag 
		&\quad + \int_0^t \langle \mu^N_s, (\nu  - \mu u^N_s\wedge A) |\cdot|^{\kappa} \rangle \, \dd s \\ \notag
		& \quad + \frac{\sqrt{2}}{N} \sum_{k \in \Lambda^N} \int_0^t  \1_{\{k \in I_s\}} \nabla  |\cdot|^{\kappa} (X^k_s) \cdot \, \dd B^k_s \\ \notag 
		&\quad + \frac1N \int_0^t \int_{\Lambda^N} \int_{\R_+} \1_{\{k \in I_s\}} |X^k_{s^-}|^{\kappa} \big(\1_{\{\rho < \nu\}} - \1_{\{\nu\leq \rho < \nu + \mu u^{N}_s(X^k_{s^-})\wedge A\}} \big)\, \NNN (\dd s,\dd k,\dd \rho).
	\end{align*}
	From Jensen's inequality, the BDG inequality for the Brownian integral and the Kunita inequality for the Poisson integral (see \cite[Theorem~4.4.23]{Applebaum}), it comes that
	\begin{align*}
		\E \sup_{v \in [0,t]}\langle \mu^N_v, |\cdot|^{\kappa} \rangle^m\notag 
		&\lesssim \E \langle \mu^N_0 , |\cdot|^{\kappa}  \rangle^m+ \int_0^t \E  |\langle \mu^N_s, \Delta  |\cdot|^{\kappa}  \rangle|^m \, \dd s
		  +  \int_0^t \E |\langle \mu^N_s, F_A(K\ast u^N_s) \cdot \nabla  |\cdot|^{\kappa} \rangle|^m \, \dd s \\
		  &\quad+ \int_0^t \E\langle \mu^N_s,|\cdot|^{\kappa} \rangle^m\, \dd s 
		 + \frac{1}{N^{m}} \E  \bigg[ \Big(\int_0^t \sum_{k \in I_s} |\nabla |\cdot|^{\kappa} |^2 (X^k_s) \, \dd s\Big)^{\frac{m}{2}} \bigg] \\
		& \quad+ \frac{1}{N^m} \sum_{l=0}^1 \E \bigg[ \Big( \int_0^t \sum_{k \in I_s} |X^k_s|^{2 \kappa (\frac{m}{2})^l} \, \dd s\Big)^{(\frac{m}{2})^{1-l}} \bigg] \\
		&\lesssim \E \langle \mu^N_0 , |\cdot|^{\kappa}  \rangle^m + \int_0^t \E  \langle \mu^N_s, |(\Delta  |\cdot|^{\kappa})| \rangle^m \, \dd s  
		 +  \int_0^t \E \langle \mu^N_s, \big| \nabla  |\cdot|^{\kappa} \big| \rangle^m \, \dd s \\ 
		 &\quad+ \int_0^t \E\langle \mu^N_s, |\cdot|^{\kappa} \rangle^m\, \dd s 
		 + \frac{1}{N^{\frac{m}{2}}} \int_0^t \E \langle \mu^N_s, \big| \nabla |\cdot|^\kappa \big|^m \rangle \, \dd s \\ 
		 &  \quad+ \frac{1}{N^m} \sum_{l=0}^1 N^{(\frac{m}{2})^{1-l}}  \int_0^t \E \langle \mu^N_s,  |\cdot|^{2 \kappa (\frac{m}{2})^l} \rangle^{(\frac{m}{2})^{1-l}} \, \dd s .
	\end{align*}
Using that $2 (\frac{m}{2})^l \geq 2$, it comes
\begin{equation*}
\langle \mu^N_s, |\cdot|^{2 \kappa (\frac{m}{2})^l} \rangle \leq \frac{1}{N} \Big(\sum_{k \in \Lambda^N} \1_{\{k \in I_s\}} |X^k_s|^{\kappa} \Big)^{2 (\frac{m}{2})^l} = N^{2 (\frac{m}{2})^l -1} \langle \mu^N_s, |\cdot|^{\kappa}\rangle^{2 (\frac{m}{2})^l},
\end{equation*}
hence we obtain 
	\begin{align*}
	\E \sup_{v \in [0,t]}\langle \mu^N_v, |\cdot|^{\kappa} \rangle^m
	&\lesssim \E \langle \mu^N_0 , |\cdot|^{\kappa}  \rangle^m + \int_0^t \E  \langle \mu^N_s, \big|\Delta  |\cdot|^{\kappa} \big| \rangle^m \, \dd s  
	+  \int_0^t \E \langle \mu^N_s, \big| \nabla  |\cdot|^{\kappa} \big| \rangle^m \, \dd s \\ 
	&\quad + \int_0^t \E\langle \mu^N_s,|\cdot|^{\kappa} \rangle^m\, \dd s 
	 + \frac{1}{N^{\frac{m}{2}}} \int_0^t \E \langle \mu^N_s, \big| \nabla |\cdot|^\kappa \big|^m \rangle \, \dd s \\ 
	 &  \quad+ \frac{1}{N^m} \sum_{l=0}^1 N^{(\frac{m}{2})^{1-l} +2 (\frac{m}{2})^l -1}  \int_0^t \E \langle \mu^N_s,  |\cdot|^{\kappa} \rangle^{2 (\frac{m}{2})^l} \, \dd s  . 
\end{align*}
Since  $(\frac{m}{2})^{1-l} +2 (\frac{m}{2})^l -1 \leq m$ for $l\in \{0,1\}$ and $m\geq 2$, it follows that
	\begin{align*}
	\E \sup_{v \in [0,t]}\langle \mu^N_v, |\cdot|^{\kappa} \rangle^m
	&\lesssim \E \langle \mu^N_0 , |\cdot|^{\kappa}  \rangle^m + \int_0^t \E  \langle \mu^N_s, \big|\Delta  |\cdot|^{\kappa} \big| \rangle^m \, \dd s  
	+  \int_0^t \E \langle \mu^N_s, \big| \nabla  |\cdot|^{\kappa} \big| \rangle^m \, \dd s \\
	&\quad + \int_0^t \E\langle \mu^N_s,|\cdot|^{\kappa} \rangle^m\, \dd s  
	 + \frac{1}{N^{\frac{m}{2}}} \int_0^t \E \langle \mu^N_s, \big| \nabla |\cdot|^\kappa \big|^m \rangle \, \dd s . 
\end{align*}
	By explicit computation of the gradient and the Laplacian of $|\cdot|^{\kappa}$ and by Proposition~\ref{KS_prop_moment_k_number_ptcles}, we deduce that 
	\begin{align*}
		\E \sup_{v \in [0,t]}\langle \mu^N_v, |\cdot|^{\kappa} \rangle^m\notag  
		&\lesssim 1+ \E \langle \mu^N_0 , |\cdot|^{\kappa}  \rangle^m +  \int_0^t \E\langle \mu^N_s,|\cdot|^{\kappa} \rangle^m\, \dd s  .
	\end{align*}
	We conclude the proof by Grönwall's inequality. 
\end{proof}

\begin{proof}[Proof of Proposition \ref{KS_estimate_Lp_Poisson_integral}]
	When $p \geq 2$, the proof follows from the continuous embedding of $H^{\gamma}_{2}$ into $L^p$ when $\gamma > \frac{d}{2} - \frac{d}{p}$, see \eqref{eq:Bessel_embedding}, and from Proposition~\ref{KS_estimate_Bessel_Poisson_integral}. Let us now focus on the case $p=1$ with $m = 2^n$, the case $p \in (1,2)$ will follow by interpolation. We emphasise that the proof of Proposition~\ref{KS_estimate_Bessel_Poisson_integral} does not carry over here because the functional BDG used therein does not hold in $L^1(\R^d)$. To bypass this difficulty, we will go from $L^1(\R^d)$ to $L^2(\R^d)$ by introducing weights.
	Hence using Cauchy-Schwarz's inequality, it holds	
	\begin{align*}
		\E \Vert J^N_t \Vert_{L^1}^m &\leq C\, \E \Vert (1+ \vert\cdot \vert)^{\frac{d+1}{2}} J^N_t \Vert_{L^2}^m \\ 
		&= C\, \E \bigg\Vert \frac1N \int_0^t \int_{\Lambda^N} \int_{0}^{\nu +\mu A} (1+ \vert\cdot \vert)^{\frac{d+1}{2}}\1_{\{k \in I_s\}} e^{(t-s)\Delta}\theta^N(\cdot-X^k_{s^-})  \\ 
		& \hspace{4cm} \times (\1_{\{\rho < \nu\}} - \1_{\{\nu\leq \rho < \nu + \mu u^{N}_s(X^k_{s^-})\wedge A\}})\, \NNN (\dd s,\dd k,\dd \rho) \bigg\Vert_{L^2}^m.
	\end{align*}
	Using the functional BDG inequality for Poisson measures \cite[Corollary $2.14$]{Hausenblas}, it follows that
	\begin{align*}
		\E \Vert J^N_t \Vert_{L^1}^m 
		 &\lesssim \sum_{l=1}^n  \E \bigg(  \sum_{k \in \Lambda^N}\int_0^t  \int_0^{\nu +\mu A}  \left\Vert \frac1N  (1+ \vert\cdot \vert)^{\frac{d+1}{2}} \1_{\{k \in I_s\}} e^{(t-s)\Delta}\theta^N(\cdot-X^k_{s^-}) \right. \\ 
		 & \hspace{4.5cm} \times \left. \big(\1_{\{\rho < \nu\}} - \1_{\{\nu\leq \rho < \nu + \mu u^{N}_s(X^k_{s^-})\wedge A\}}\big) \right\Vert_{L^2}^{2^l}\, \dd s  \,\dd \rho \bigg)^{2^{n-l}} \\ 
		 &\lesssim \frac{1}{N^m} \sum_{l=1}^n  \E \Big(  \sum_{k \in \Lambda^N}\int_0^t    \left\Vert   (1+ \vert\cdot \vert)^{\frac{d+1}{2}} \1_{\{k \in I_s\}} e^{(t-s)\Delta}\theta^N(\cdot-X^k_{s^-}) \right\Vert_{L^2}^{2^l}\, \dd s   \Big)^{2^{n-l}} .
	\end{align*}
	By the inequality $(1+ |x+y|) \leq (1+|x|)(1+|y|)$ and Fubini's Theorem, we obtain  
	\begin{align*}
		&\E \Vert J^N_t \Vert_{L^1}^m  \nonumber \\
		&\lesssim \frac{1}{N^m} \sum_{l=1}^n  \E \Big( \sum_{k \in \Lambda^N}\int_0^t   (1+ \vert X^k_{s^-} \vert)^{(d+1)2^{l-1}} \1_{\{k \in I_s\}}   \left\Vert   (1+ \vert \cdot \vert)^{\frac{d+1}{2}} e^{(t-s)\Delta}\theta^N \right\Vert_{L^2}^{2^l}\, \dd s \Big)^{2^{n-l}}  \nonumber \\ 
		& \lesssim \frac{1}{N^m} \sum_{l=1}^n N^{2^{n-l}} \E \sup_{s \in [0,t]} |\langle \mu^N_s, (1+ |\cdot|)^{(d+1)2^{l-1}} \rangle|^{2^{n-l}} \left( \int_0^t \left\Vert   (1+ \vert \cdot \vert)^{\frac{d+1}{2}} e^{(t-s)\Delta}\theta^N \right\Vert_{L^2}^{2^l} \, \dd s \right)^{2^{n-l}} .
	\end{align*}
	Lemma \ref{Lemma_control_moment_ptcles} yields
	\begin{align}\label{KS_proof_L1_eq2}
		\E \Vert J^N_t \Vert_{L^1}^m 
		& \lesssim \frac{1}{N^m} \sum_{l=1}^n N^{2^{n-l}} \left( \int_0^t \big\Vert (1+ \vert \cdot \vert)^{\frac{d+1}{2}} e^{(t-s)\Delta}\theta^N \big\Vert_{L^2}^{2^l} \, \dd s \right)^{2^{n-l}}.
	\end{align}
	Let us consider the $L^2$ term in \eqref{KS_proof_L1_eq2}. By a change of variable and Fubini's Theorem, we have
	\begin{align*}
		\int_{\R^d} (1+|y|)^{d+1} |e^{(t-s)\Delta}\theta^N (y)|^2 \, \dd y 
		&= \int_{\R^d} (1+|y|)^{d+1} \left\vert \int_{\R^d} g_{2(t-s)}(y-x) \theta^N(x)  \, \dd x \right\vert^2 \, \dd y \\
		& \leq \int_{\R^d} (1+|y|)^{d+1}  \int_{\R^d} g_{2(t-s)}(y-x) |\theta^N(x)|^2  \, \dd x  \, \dd y\\
		& =  N^{\alpha d} \int_{\R^d} (1+|y|)^{d+1}  \int_{\R^d} g_{2(t-s)}\big(y-\frac{z}{N^{\alpha}}\big) |\theta(z)|^2  \, \dd z  \, \dd y
		\\& = N^{\alpha d} \int_{\R^d} |\theta(z)|^2   \int_{\R^d} g_{2(t-s)}(y)\left(1+\left\vert y+\frac{z}{N^{\alpha}}\right\vert\right)^{d+1}   \, \dd y \, \dd z \\ &\lesssim N^{\alpha d}  \int_{\R^d} |\theta(z)|^2   \int_{\R^d} g_{2(t-s)}(y)\left(1+ |y|^{d+1} +\left\vert \frac{z}{N^{\alpha}}\right\vert^{d+1} \right)\, \dd y \, \dd z \\ 
		&\lesssim N^{\alpha d} (1 + (t-s)^{\frac{d+1}{2}} + N^{-\alpha(d+1)}) \\ 	
		&\lesssim N^{\alpha d}.
	\end{align*}
	Following the same lines but using Jensen's inequality for the probability measure $\theta^N \, \dd x$, we obtain that 
	\begin{align*}
		\int_{\R^d} (1+|y|)^{d+1} |e^{(t-s)\Delta}\theta^N (y)|^2 \, \dd y 
		& \leq \int_{\R^d} (1+|y|)^{d+1}  \int_{\R^d} |g_{2(t-s)}(y-x)| \theta^N(x)  \, \dd x  \, \dd y \\
		& = \int_{\R^d} (1+|y|)^{d+1} \int_{\R^d} \frac{1}{(t-s)^{\frac{d}{2}}} (g_1(z))^2 \theta^N(y- \sqrt{t-s}\,z) \, \dd z\, \dd y \\ 
		&\lesssim \frac{1}{(t-s)^{\frac{d}{2}}} \int_{\R^d} (1 + |y|^{d+1} + (\sqrt{t-s}|z|)^{d+1}) (g_1(z))^2 \theta^N(y) \,\dd z\, \dd y \\ 
		&\lesssim  (t-s)^{-\frac{d}{2}}.
	\end{align*}
	By interpolation of the two preceding estimates, we find that for any $\delta \in [0,1]$,
	 \begin{equation*}
	 \int_{\R^d} (1+|y|)^{d+1} |e^{(t-s)\Delta}\theta^N (y)|^2 \, \dd y \lesssim (t-s)^{-\delta\frac{d}{2}}N^{\alpha d(1-\delta)}.
	 \end{equation*}
Let $l \in \{1,\dots n\}$. Plugging the previous estimate in \eqref{KS_proof_L1_eq2}, we obtain that for any $\eps\in (0,1)$, choosing $\delta = \frac{4}{d2^l}(1-\eps)$, there exists $C\equiv C_{\eps}>0$ such that
	\begin{align*}
		\E \Vert J^N_t \Vert_{L^1}^m  
		&\leq \frac{C}{N^m} \sum_{l=1}^n N^{2^{n-l}} N^{\alpha d\left(1 - \frac{4}{d2^l}(1-\eps)\right)2^{l-1} 2^{n-l}} \\ 
		&\leq \frac{C}{N^m} N^{\frac{m}{2}} N^{\alpha d \left(1 - \frac{4}{d m}(1-\eps)\right) \frac{m}{2}}.
	\end{align*}
Thus, for any $\eps >0$ there exists $C>0$ such that for any $t \in [0,T]$,\begin{equation*}
	 \big\Vert \Vert J^N_t \Vert_{L^1}\big\Vert_{L^m_{\Omega}} 
	\leq CN^{-\frac12 \left(1 - \alpha (d - \frac{4}{m})\right) + \eps }.
\end{equation*}
\end{proof}

\subsection{\emph{A priori} bounds on the mollified empirical measure}
\label{subsec:apriorimollempmeas}

In this section, we provide an \emph{a priori} bound on the mollified empirical measure $u^N$, see Proposition~\ref{Prop_bounds_Bessel_norm_mollified}. Beforehand, we state the following Gr\"onwall's Lemma with singular convolution kernels, taken from \cite{Webb} (see Corollary 3.3 therein), which will be useful several times in the following. We recall it here for the reader's convenience.
\begin{lemma}\label{lem:Gronwall}
Let $\beta_{1},\beta_{2}\geq 0$ such that $\beta_{1}+\beta_{2}<1$ and let $a,b$ be bounded measurable, non-negative functions. Assume that a bounded measurable function $f:[0,T] \to \R_{+}$ satisfies
\begin{equation*}
f(t) \leq a(t) + b(t) \int_{0}^t s^{-\beta_{1}} (t-s)^{-\beta_{2}}\, f(s)\, \dd s \quad \text{for a.e.}~ t\in [0,T].
\end{equation*}
Then there exists $C=C_{\beta_{1},\beta_{2}}>0$ such that
\begin{equation*}
f(t) \leq C \esssup_{s\leq t} a(s)\ \exp\big(C \esssup_{s\leq t} b(s)\, t^{1-\beta_{1}} \big) \quad \text{for a.e.}~ t\in [0,T].
\end{equation*}
\end{lemma}

\begin{proposition}
\label{Prop_bounds_Bessel_norm_mollified}
Let $r\in (d,+\infty)$ and $\gamma_{0} \in [0,1)$ as in Assumption \ref{Assumptions_Init}, \emph{i.e.} such that $\sup_{N\geq 1}  \E \Vert u^N_0 \Vert_{\gamma_{0},r}^m <\infty$ for any $m\geq 1$. 
Let $\eta \geq 0$, $\bar\eta \coloneqq (\eta-\gamma_{0})\vee 0$ and $p \in [r,+\infty]$ such that the following conditions hold:
\begin{equation*}
0 < \alpha < \frac{1}{2\big( d +  \eta - \frac{d}{p} \big)} \quad \text{and} \quad   1+\eta + \bar\eta + d\Big(\frac{1}{r}-\frac{1}{p}\Big) <2.
\end{equation*}
Then we have, for any $n\in \N$ and $m=2^n$,
	\begin{equation*}
	\sup_{N \geq 1}\sup_{t \in [0,T]} (1\wedge t)^{\frac{1}{2} \left( \bar\eta + d(\frac{1}{r}-\frac{1}{p}) \right)}  \big\Vert  u^N_t \big\Vert_{L^m(\Omega; H^\eta_{p})}  < +\infty.
	\end{equation*}
\end{proposition}

\begin{proof}
Define $U(t) \coloneqq (1\wedge t)^{\frac{1}{2} \left( \bar\eta + d(\frac{1}{r}-\frac{1}{p}) \right)} \Vert(I - \Delta)^{\frac{\eta}{2}}u^N_t \Vert_{L^p}$. 
Applying the operator $(I - \Delta)^{\frac{\eta}{2}}$ to the mild formulation \eqref{KS_mollified_empirical_mesure_equation_mild}, we obtain that
\begin{align}\label{eq:decomp-Ut}
	U(t) 
	&\leq (1\wedge t)^{\frac{1}{2} \left( \bar\eta + d(\frac{1}{r}-\frac{1}{p}) \right)} \big\Vert(I - \Delta)^{\frac{\eta}{2}} e^{t\Delta}u^N_0 \big\Vert_{L^p} \nonumber\\
	&\quad + (1\wedge t)^{\frac{1}{2} \left( \bar\eta + d(\frac{1}{r}-\frac{1}{p}) \right)}  \chi \left\Vert \int_0^t (I - \Delta)^{\frac{\eta}{2}} \nabla \cdot e^{(t-s)\Delta} \langle \mu^N_s, F_A(K\ast u^N_s(\cdot)) \theta^N(x-\cdot) \rangle \, \dd s \right\Vert_{L^p} \nonumber\\
	&\quad + (1\wedge t)^{\frac{1}{2} \left( \bar\eta + d(\frac{1}{r}-\frac{1}{p}) \right)} \left\Vert\int_0^t (I - \Delta)^{\frac{\eta}{2}} e^{(t-s)\Delta}\langle \mu^N_s, (\nu  - \mu u^N_s(\cdot)\wedge A)\theta^N(x-\cdot) \rangle \, \dd s \right\Vert_{L^p} \nonumber\\ 
	& \quad + (1\wedge t)^{\frac{1}{2} \left( \bar\eta + d(\frac{1}{r}-\frac{1}{p}) \right)} \Big(\Vert (I - \Delta)^{\frac{\eta}{2}} M^N_t \Vert_{L^p} + \Vert (I - \Delta)^{\frac{\eta}{2}}J^N_t \Vert_{L^p} \Big)\nonumber\\ 
	&\eqqcolon I_1 + I_2 + I_3 + I_{4}.
\end{align}
Use \eqref{eq:Besov-heat-estimate} to get 
\begin{align*}
I_{1} &= (1\wedge t)^{\frac{1}{2} \left( \bar\eta + d(\frac{1}{r}-\frac{1}{p}) \right)} \big\Vert (I - \Delta)^{\frac{\eta-\gamma_{0}}{2}} e^{t\Delta} (I - \Delta)^{\frac{\gamma_{0}}{2}} u^N_0 \big\Vert_{L^p}  \\
 &\lesssim \lVert u^N_{0} \rVert_{\gamma_{0},r}.
\end{align*}
For $I_2$, use successively the triangle inequality, the Gaussian estimate \eqref{eq:Besov-heat-estimate}, the boundedness of $F_{A}$ and the Bessel embedding \eqref{eq:Bessel_embedding} for $\eta\geq 0$, to get
\begin{align}\label{eq:boundUt-I2}
I_2 & \lesssim (1\wedge t)^{\frac{1}{2} \left( \bar\eta + d(\frac{1}{r}-\frac{1}{p}) \right)} \int_{0}^t (t-s)^{-\frac{1+\eta}{2}}  \lVert \langle \mu^N_s, F_A(K\ast u^N_s(\cdot)) \theta^N(x-\cdot) \rangle \rVert_{L^p}\, \dd s \nonumber\\
&\lesssim (1\wedge t)^{\frac{1}{2} \left( \bar\eta + d(\frac{1}{r}-\frac{1}{p}) \right)}  \int_0^t (t-s)^{-\frac{1+\eta}{2}} \Vert u^N_s \Vert_{L^p} \, \dd s \nonumber\\
&\lesssim (1\wedge t)^{\frac{1}{2} \left( \bar\eta + d(\frac{1}{r}-\frac{1}{p}) \right)}  \int_0^t (t-s)^{-\frac{1+\eta}{2}} \Vert u^N_s \Vert_{\eta,p} \, \dd s \nonumber\\
&\lesssim \int_0^t (t-s)^{-\frac{1+\eta}{2}} (1\wedge s)^{-\frac{1}{2} \left( \bar\eta + d(\frac{1}{r}-\frac{1}{p}) \right)} \, U(s)\, \dd s .
\end{align}
Now for $I_3$, use successively the triangle inequality, the Gaussian estimate \eqref{eq:Besov-heat-estimate}, the boundedness of $(\nu  - \mu u^N_s(\cdot)\wedge A)$ and the Bessel embedding \eqref{eq:Bessel_embedding} for $\eta\geq 0$,  to write 
\begin{align}\label{eq:boundUt-I3}
	I_{3}&\leq \left\Vert\int_0^t (I - \Delta)^{\frac{\eta}{2}} e^{(t-s)\Delta}\langle \mu^N_s, (\nu  - \mu u^N_s(\cdot)\wedge A)\theta^N(x-\cdot) \rangle \, \dd s \right\Vert_{L^p} \nonumber\\
	& \lesssim \int_0^t (t-s)^{-\frac{\eta}{2}} \left\Vert \langle \mu^N_s, (\nu  - \mu u^N_s(\cdot)\wedge A)\theta^N(x-\cdot) \rangle \right\Vert_{L^p}  \dd s \nonumber\\ 
	&\lesssim \int_0^t (t-s)^{-\frac{\eta}{2}} (1\wedge s)^{-\frac{1}{2} \left( \bar\eta + d(\frac{1}{r}-\frac{1}{p}) \right)} \, U(s)\, \dd s .
\end{align}
Then for $I_{4}$, we first use the Bessel embedding \eqref{eq:Bessel_embedding} with $p\geq 2$ (since $p\geq r> d$ and $d\geq 2$); it then remains to apply Propositions~\ref{KS_estimate_Bessel_Brownian_integral} and \ref{KS_estimate_Bessel_Poisson_integral}, which reads
\begin{align}\label{eq:boundUt-I4}
\E[(I_{4})^m] &\lesssim (1\wedge t)^{\frac{1}{2} \left( \eta + d(\frac{1}{r}-\frac{1}{p}) \right)} \Big( \E \Vert M^N_t \Vert_{\eta+d(\frac{1}{2}-\frac{1}{p}),2}^m + \E \Vert J^N_t \Vert_{\eta+d(\frac{1}{2}-\frac{1}{p}),2}^m \Big) \nonumber\\
&\lesssim N^{-\frac{1}{2} \big(1-\alpha\big(d+2(\eta+d(\frac{1}{2}-\frac{1}{p}))\big)\big)+\varepsilon}
+ N^{-\frac{1}{2} \big(1-\alpha\big(d+2(\eta+d(\frac{1}{2}-\frac{1}{p}))-\frac{4}{m}\big)\big)+\varepsilon} \nonumber\\
&\lesssim 1,
\end{align}
using the assumption $\alpha < \frac{1}{2( d +  \eta - \frac{d}{p} )}$.

Hence, in view of \eqref{eq:decomp-Ut}-\eqref{eq:boundUt-I2}-\eqref{eq:boundUt-I3} and \eqref{eq:boundUt-I4}, we get that there exists $C = C_{\eta,r,p,d,A,\mu,\nu,T}$ such that
\begin{align*}
	\lVert U(t) \rVert_{L^m_{\Omega}} \leq C \Big(1+ \lVert u^N_{0} \rVert_{L^m(\Omega;H^{\gamma_{0}}_{r})} + \int_0^t (t-s)^{-\frac{1+\eta}{2}} (1\wedge s)^{-\frac{1}{2} \left( \bar\eta + d(\frac{1}{r}-\frac{1}{p}) \right)} \, \lVert U(s) \rVert_{L^m_{\Omega}}\, \dd s \Big).
\end{align*}
Since we assumed $\sup_{N\geq 1} \E \lVert u^N_{0} \rVert_{\gamma_{0},r}^m <\infty$ and $\frac{1+\eta}{2} + \frac{1}{2}( \bar\eta + d(\frac{1}{r}-\frac{1}{p}) )<1$, Grönwall's lemma for convolution integrals, see Lemma~\ref{lem:Gronwall}, allows us to conclude the proof.
%
\end{proof}

\subsection{Convergence of the mollified empirical measure: Proof of Theorem~\ref{th:convergence_moderate}}
\label{subsec:proofThm2}

Using the mild formulation \eqref{eq:KSmild} of $u$ and \eqref{KS_mollified_empirical_mesure_equation_mild} for $u^N$, we  write 
\begin{align*}
	u^N_t(x) - u_t(x) &= e^{t\Delta}(u^N_0 - u_0)(x) \\ 
	&\quad- \chi \int_0^t \nabla \cdot e^{(t-s)\Delta} \left[ \langle \mu^N_s, F_A(K\ast u^N_s(\cdot)) \theta^N(x-\cdot) \rangle - (K\ast u_s)(x) u_s(x)\right]  \dd s  \\
	 &\quad + \int_0^t e^{(t-s)\Delta} \left[ \langle \mu^N_s, (\nu  - \mu u^N_s(\cdot)\wedge A)\theta^N(x-\cdot) \rangle - (\nu u_s(x) - \mu u_s^2(x)) \right] \dd s \\
	  & \quad - M^N_t + J^N_t.
\end{align*}
By the triangle inequality, we deduce that for any $t \in [0,T]$,
\begin{align}\label{Eq:proof_cv_decomposition}
	\Vert u^N_t - u_t \Vert_{L^1 \cap L^r} \notag&\leq  \Vert u^N_0 - u_0\Vert_{L^1 \cap L^r} \\ \notag 
	&\quad +  \chi \int_0^t  \big\Vert \nabla \cdot e^{(t-s)\Delta} \left[ \langle \mu^N_s, F_A(K\ast u^N_s(\cdot)) \theta^N(x-\cdot) \rangle - (K\ast u_s)(x) u_s(x)\right] \big\Vert_{L^1 \cap L^r} \, \dd s  \\ \notag
	&\quad + \int_0^t  \big\Vert  e^{(t-s)\Delta} \left[ \langle \mu^N_s, (\nu  - \mu u^N_s(\cdot)\wedge A)\theta^N(x-\cdot) \rangle - (\nu u_s(x) - \mu u_s^2(x)) \right] \big\Vert_{L^1 \cap L^r} \, \dd s \\ \notag 
	& \quad + \Vert  M^N_t\Vert_{L^1 \cap L^r} + \Vert J^N_t \Vert_{L^1 \cap L^r} \\ 
	&\eqqcolon I_1 + I_2 + I_3 + I_4. 
\end{align}

Let us focus first on $I_2$, which is treated similarly in the proof of \cite[Theorem~1.3]{Pisa}, but we give the proof here for the sake of completeness. We decompose $I_2$ as follows:
\begin{align}\label{Eq:decomposition_interaction_term}
I_2 &\leq \notag\chi \int_0^t  \Big\Vert \nabla \cdot e^{(t-s)\Delta} \left[ \langle \mu^N_s, (F_A(K\ast u^N_s(\cdot)) - F_A(K\ast u^N_s(x))) \theta^N(x-\cdot) \rangle \right] \Big\Vert_{L^1 \cap L^r} \, \dd s \\ 
& \quad + \chi \int_0^t  \Big\Vert \nabla \cdot e^{(t-s)\Delta} \left[ F_A(K\ast u^N_s)u^N_s -F_A(K\ast u_s)u_s\right] \Big\Vert_{L^1\cap L^r} \, \dd s \\ \notag 
&\eqqcolon I_{2,A} + I_{2,B},
\end{align}
since $A$ is chosen large enough to ensure that for any $s \in [0,T]$, $F_A(K\ast u_s) = K\ast u_s$. In view of the heat kernel estimate \eqref{eq:heat-estimate}, we have
\begin{equation*}
I_{2,A} \leq C \int_0^t  \frac{1}{\sqrt{t-s}} \big\Vert \langle \mu^N_s, \big(F_A(K\ast u^N_s(\cdot)) - F_A(K\ast u^N_s(x))\big) \theta^N(x-\cdot) \rangle \big\Vert_{L^1 \cap L^r} \, \dd s .
\end{equation*}
Using the Lipschitz continuity of $F_A$ and the $\zeta$-Hölder regularity of $K\ast u^N$ given in Proposition~\ref{KS_properties_kernel}\ref{item:HolderboundK} for $\zeta= 1-\frac{d}{r}$, we deduce that 
\begin{align}\label{eq:rateKu}
	I_{2,A} &\leq C \int_0^t  \frac{\Vert u^N_s\Vert_{L^1\cap L^r}}{\sqrt{t-s}}\big\Vert \langle \mu^N_s, |x-\cdot|^{\zeta}\theta^N(x-\cdot) \rangle \big\Vert_{L^1 \cap L^r} \, \dd s.
\end{align}
Since the support of $\theta$ is contained in the unit ball, we have for all $x,y \in \R^d$,
\begin{equation}\label{inequality_holder_thetaN}
|x-y|^{\zeta} \theta^N(x-y) \leq N^{-\zeta\alpha} \theta^N(x-y).
\end{equation}
It follows that 
\begin{align*}
	I_{2,A} &\leq CN^{-\zeta\alpha} \int_0^t  \frac{\Vert u^N_s\Vert_{L^1\cap L^r}^2}{\sqrt{t-s}}\, \dd s.
\end{align*}
Hölder's inequality for $p = \frac32$ yields 
\begin{align}\label{Eq:first_term_decomposition_interaction}
	I_{2,A} \notag&\leq CN^{-\zeta\alpha} \left( \int_0^t (t-s)^{-\frac34}\,  \dd s\right)^{\frac23}  \left(\int_0^t  \Vert u^N_s\Vert_{L^1\cap L^r}^6\, \dd s\right)^{\frac13} \\ &\leq CN^{-\zeta\alpha} \left(\int_0^t  \Vert u^N_s\Vert_{L^1\cap L^r}^6\, \dd s\right)^{\frac13}.
\end{align}
Let us now treat $I_{2,B}$ defined in \eqref{Eq:decomposition_interaction_term}. Using \eqref{eq:heat-estimate}, we have 
\begin{align*} 
I_{2,B} 
&\leq C\int_0^t \frac{1}{\sqrt{t-s}} \big\Vert F_A(K\ast u^N_s)u^N_s -F_A(K\ast u_s)u_s \big\Vert_{L^1\cap L^r} \, \dd s \\ 
&\leq C\int_0^t \frac{1}{\sqrt{t-s}} \Big(\big\Vert  F_A(K\ast u^N_s)(u^N_s-u_s)\big\Vert_{L^1\cap L^r} +  \big\Vert  (F_A(K\ast u^N_s) - F_A(K\ast u_s))u_s\big\Vert_{L^1\cap L^r} \Big)\, \dd s.
\end{align*}
Using the boundedness of $F_A$ and its Lipschitz continuity, we get  \begin{align*} 
I_{2,B} 
&\leq C\int_0^t \frac{1}{\sqrt{t-s}} \left(\Vert u^N_s-u_s\Vert_{L^1\cap L^r} + \Vert K\ast (u^N_s - u_s) \Vert_{L^{\infty}} \Vert  u_s\Vert_{L^1\cap L^r}\,  \right) \dd s.
\end{align*}
It follows from Proposition~\ref{KS_properties_kernel} and the fact that $u\in L^\infty([0,T];L^1\cap L^\infty)$, see Theorem~\ref{Thm_WP_KS}, that 
\begin{align*}
I_{2,B} 
&\leq C\int_0^t \frac{1}{\sqrt{t-s}} \Vert u^N_s-u_s\Vert_{L^1\cap L^r}\, \dd s.
\end{align*}
Plugging this estimate and \eqref{Eq:first_term_decomposition_interaction} into \eqref{Eq:decomposition_interaction_term}, we have proved that
\begin{equation}\label{Eq:proof_thm_cv_I2_estimate}
	I_2 \lesssim  \int_0^t \frac{1}{\sqrt{t-s}} \Vert u^N_s - u_s \Vert_{L^1\cap L^r} \, \dd s + N^{-\zeta \alpha} \left(\int_0^t \Vert u^N_s \Vert_{L^1 \cap L^r}^6 \, \dd s \right)^{\frac{1}{3}} .
\end{equation}
Let us now treat $I_3$. By a convolution inequality, one has 
\begin{align*}
	I_3 & \leq \int_0^t  \Big(\nu \Vert u^N_s - u_s \Vert_{L^1 \cap L^r} + \mu \big\Vert \langle \mu^N_s,   (u^N_s(\cdot)\wedge A) \theta^N(x-\cdot) \rangle  -  u_s^2 \big\Vert_{L^1 \cap L^r}\Big) \, \dd s \\ 
	&\leq  \nu \int_0^t   \Vert u^N_s - u_s \Vert_{L^1 \cap L^r} \, \dd s + \mu \int_0^t  \big\Vert \langle \mu^N_s, (u^N_s(\cdot)\wedge A - u^N_s(x)\wedge A) \theta^N(x-\cdot) \rangle \big\Vert_{L^1 \cap L^r} \, \dd s \\ 
	&\quad  + \mu \int_0^t \big\Vert (u^N_s\wedge A) u^N_s -  u_s^2 \big\Vert_{L^1 \cap L^r} \, \dd s \\ 
	&\leq  \nu \int_0^t \Vert u^N_s - u_s \Vert_{L^1 \cap L^r} \, \dd s + \mu \int_0^t  \big\Vert \langle \mu^N_s,   (u^N_s(\cdot)\wedge A - u^N_s(x)\wedge A) \theta^N(x-\cdot) \rangle \big\Vert_{L^1 \cap L^r} \, \dd s \\ 
	&\quad  + \mu \int_0^t \big\Vert (u^N_s\wedge A)( u^N_s -  u_s)  \big\Vert_{L^1 \cap L^r} \, \dd s + \mu \int_0^t \big\Vert (u^N_s\wedge A - u_s) u_s \big\Vert_{L^1 \cap L^r} \, \dd s.
\end{align*}
Recalling that $A\geq A_{T}$, see \eqref{Def_cutoff_limite}, so that $\sup_{s \in [0,T]}\Vert u_s \Vert_{L^{\infty}} \leq A$, and using Theorem~\ref{Thm_WP_KS} we obtain 
\begin{align*}
	I_3 &\leq  C \int_0^t   \Vert u^N_s - u_s \Vert_{L^1 \cap L^r} \, \dd s + \mu \int_0^t  \big\Vert \langle \mu^N_s,   (u^N_s(\cdot)\wedge A - u^N_s(x)\wedge A) \theta^N(x-\cdot) \rangle \big\Vert_{L^1 \cap L^r} \, \dd s.
\end{align*}
Recall that $r>d$ and $\gamma\in(\frac{d}{r},1)$ satisfy Assumption~\ref{Assumptions_Init}, and let $\tilde{\zeta} \coloneqq \gamma-\frac{d}{r}$. It follows from \eqref{inequality_holder_thetaN} that 
\begin{align}\label{eq:rateu2}
 |\langle \mu^N_s,   (u^N_s(\cdot)\wedge A - u^N_s(x)\wedge A) \theta^N(x-\cdot) \rangle| 
 &= \frac1N \bigg\vert\sum_{k \in I_s} \theta^N(x-X^k_s) (u^N_s(X^k_s)\wedge A - u^N_s(x)\wedge A) \bigg\vert \nonumber\\ 
 &\leq \frac1N\sum_{k \in I_s} \theta^N(x-X^k_s) |x-X^k_s|^{\tilde{\zeta}} \, [ u^N_s]_{\tilde{\zeta}} \nonumber\\ 
 &\leq N^{-\tilde{\zeta}\alpha} u^N_s(x)\, [ u^N_s]_{\tilde{\zeta}}.
\end{align}
We thus have 
\begin{align*}
	 I_3 
	 &\leq  C \int_0^t  \Vert u^N_s - u_s \Vert_{L^1 \cap L^r} \, \dd s + \mu N^{-\tilde{\zeta} \alpha} \int_0^t \Vert u^N_s\Vert_{L^1 \cap L^r} \, [ u^N_s]_{\tilde{\zeta}} \, \dd s.
\end{align*}
In view of \eqref{eq:embedHolder}, $H^{{\gamma}}_{r}(\R^d)$ is continuously embedded into $\CC^{\tilde{\zeta}}(\R^d)$. Hence it follows that
\begin{align}\label{Eq:proof_thm_cv_I3_estimate}
	I_3 &\lesssim  \int_0^t   \Vert u^N_s - u_s \Vert_{L^1 \cap L^r} \, \dd s + N^{-\tilde{\zeta}\alpha} \int_0^t  \Vert u^N_s\Vert_{L^1 \cap L^r} \Vert u^N_s \Vert_{{\gamma},r} \, \dd s .
\end{align}
Coming back to \eqref{Eq:proof_cv_decomposition} and using \eqref{Eq:proof_thm_cv_I2_estimate}, \eqref{Eq:proof_thm_cv_I3_estimate} and the Grönwall lemma for convolution integrals (Lemma~\ref{lem:Gronwall}), we get that for any $N \geq 1$ and $t \in [0,T]$,
\begin{align*}
	\Vert u^N_t -u_t \Vert_{L^1 \cap L^r} 
	&\lesssim \Vert u^N_0 - u_0 \Vert_{L^1 \cap L^r} + N^{-\zeta \alpha} \left(\int_0^t \Vert u^N_s \Vert_{L^1 \cap L^r}^6 \, \dd s \right)^{\frac{1}{3}} \\ 
	&\qquad+  N^{-\tilde{\zeta}\alpha} \int_0^t \Vert u^N_s\Vert_{L^1 \cap L^r} \Vert u^N_s \Vert_{\gamma,r} \, \dd s   + \Vert  M^N_t\Vert_{L^1 \cap L^r} + \Vert J^N_t \Vert_{L^1 \cap L^r}.
\end{align*}
Since for any $x \geq 0$, $x^{\frac13} \leq 1+x$ and owing to the triangle and the Cauchy-Schwarz inequalities, we deduce that for any $t \in [0,T]$,
\begin{align*}
	\big\Vert \Vert u^N_t -u_t \Vert_{L^1 \cap L^r}\Vert_{L^m_\Omega} 
	&\lesssim \big\Vert \Vert u^N_0 - u_0 \Vert_{L^1 \cap L^r}\big\Vert_{L^m_\Omega} + N^{-\zeta \alpha} \left(1+ \int_0^t \big(\E \Vert u^N_s \Vert_{L^1 \cap L^r}^{6m}\big)^{\frac{1}{m}} \, \dd s \right) \\ 
	&\hspace{1cm}+  N^{-\tilde{\zeta}\alpha} \left(\int_0^t  \left(\E \Vert u^N_s\Vert_{L^1 \cap L^r}^{2m} \right)^{\frac{1}{2m}} \left( \E \Vert u^N_s \Vert_{\gamma,r}^{2m}\right)^{\frac{1}{2m}} \, \dd s\right) \\ 
	& \hspace{1cm} + \big\Vert \Vert  M^N_t\Vert_{L^1 \cap L^r}\big\Vert_{L^m_\Omega} + \big\Vert \Vert J^N_t \Vert_{L^1 \cap L^r}\big\Vert_{L^m_\Omega}.
\end{align*}
By Proposition \ref{KS_prop_moment_k_number_ptcles} and Proposition \ref{Prop_bounds_Bessel_norm_mollified} applied first with $\eta =0$ and $p = r$ and then with $\eta = \gamma$, $\bar\eta = (\gamma-\gamma_{0})\vee 0$ and $p = r$, we deduce that 
\begin{align*}
\big\Vert \Vert u^N_t -u_t \Vert_{L^1 \cap L^r}\big\Vert_{L^m_\Omega} 
&\lesssim  \big\Vert \Vert u^N_0 - u_0 \Vert_{L^1 \cap L^r}\big\Vert_{L^m_\Omega} + N^{-\zeta \alpha} + N^{-\tilde{\zeta}\alpha} \int_{0}^t (1\wedge s)^{-\frac{1}{2}\bar\eta} \, \dd s  \\ 
& \quad + \big\Vert \Vert  M^N_t\Vert_{L^1 \cap L^r}\big\Vert_{L^m_\Omega} + \big\Vert \Vert J^N_t \Vert_{L^1 \cap L^r}\big\Vert_{L^m_\Omega} .
\end{align*}
In the previous expression, we have $\int_{0}^t (1\wedge s)^{-\frac{1}{2}\bar\eta} \, \dd s <\infty$ in view of Assumption~\ref{Assumptions_Init}.
Finally, we apply Proposition \ref{KS_estimate_Lp_Brownian_integral} and Proposition \ref{KS_estimate_Lp_Poisson_integral} with $p=1$ and $p=r$ (and $\tilde{m}\geq m$ for some $\tilde m$ which can be written as a power of two). Observe also that $\zeta\geq \tilde \zeta$, so that $N^{-\zeta \alpha}\leq N^{-\tilde\zeta \alpha}$. All this ensures that for any $\eps >0$, there exists $C>0$ such that for any $N \geq 1$ and any $t \in [0,T]$,
\begin{align*}
\big\Vert \Vert u^N_t -u_t \Vert_{L^1 \cap L^r}\big\Vert_{L^m_\Omega} 
&\leq C \left( \big\Vert \Vert u^N_0 - u_0 \Vert_{L^1 \cap L^r}\big\Vert_{L^m_\Omega} + N^{-\tilde{\zeta}\alpha} + N^{-\frac12 \left( 1 - 2\alpha d (1 - \frac1r )\right) + \eps}\right),
\end{align*}
which concludes the proof.

\subsection{Convergence of the empirical measure: Proof of Corollary~\ref{cor:KR-convergence}}
\label{subsec:proofCor}

Let $t \in [0,T]$. We have 
	\begin{align*}
	 \big\Vert \Vert \mu^N_t - u_t \Vert_0 \big\Vert_{L^m_\Omega} 
	 & \leq  \big\Vert \Vert \mu^N_t - u^N_t \Vert_0 \big\Vert_{L^m_\Omega} + \big\Vert \Vert u^N_t - u_t \Vert_0 \big\Vert_{L^m_\Omega}  \\ 
	 & \leq  \big\Vert \Vert \mu^N_t - u^N_t \Vert_0 \big\Vert_{L^m_\Omega} + \big\Vert \Vert u^N_t - u_t \Vert_{L^1}\big\Vert_{L^m_\Omega}.
	\end{align*}
	A simple computation proves that, for $\phi : \R^d \to \R$ measurable and bounded, there is 
	\begin{align*}
	|\langle \mu^N_t - u^N_t,\phi \rangle |  
	&\leq \Big\langle \mu^N_t, \int_{\R^d} \theta(y) \big| \phi(\cdot) - \phi (yN^{-\alpha} -\cdot)\big|\, \dd y\Big\rangle  \\ 
	&\lesssim \Vert \phi \Vert_{\text{Lip}}N^{-\alpha}\langle \mu^N_t,1\rangle.
	\end{align*}
	It follows from \eqref{eq:defKantorovich} and Proposition~\ref{KS_prop_moment_k_number_ptcles} that 
	$$\big\Vert \Vert \mu^N_t - u^N_t \Vert_0 \big\Vert_{L^m_\Omega} \lesssim N^{-\alpha} \lesssim N^{-\varrho +\eps},$$
	using the definition of $\varrho$ in Theorem~\ref{th:convergence_moderate}.


\bibliographystyle{amsplainhyper_m}
\bibliography{Branching-KS}

\end{document}